\documentclass[12pt,article]{amsart}
\usepackage{mathrsfs}
\usepackage{amssymb}
\usepackage{amsfonts}
\usepackage{amsbsy}
\usepackage{latexsym}
\usepackage{amssymb,latexsym,amsmath,amsthm}
\usepackage{framed}
\usepackage{graphicx}
\usepackage{epsfig,psfrag}
\usepackage{epstopdf}
\usepackage{xcolor}
\theoremstyle{plain}

 \theoremstyle{remark} 

\newtheorem {theo} {\bf Theorem} [section]

\newtheorem {lem} [theo] {\bf Lemma}
\newtheorem {note} [theo] {\bf Note}

\newtheorem{rem}{\bf Remark}[section]
\newtheorem{comp} {\bf Comparison}[section]

\numberwithin{equation}{section}
\begin{document}
\title[Nonuniform sampling and  approximation]{Nonuniform sampling and  approximation in Sobolev space from the perturbation of framelet system}
\author{Youfa Li}
\address{College of Mathematics and Information Science\\
Guangxi University,  Nanning, China }
\email{youfalee@hotmail.com}
\author{Deguang Han}
\address{Department of Mathematics\\
University of Central Florida\\ Orlando, FL 32816}
\email{deguang.han@ucf.edu}
\author{Shouzhi Yang}
\address{Department of Mathematics\\
University of Shantou\\ Shantou, China}
\email{szyang@stu.edu.cn}
\author{Ganji Huang}
\address{College of Mathematics and Information Science\\
Guangxi University,  Naning, China }
\email{gjhuang@gxu.edu.cn}
\thanks{Youfa Li is partially supported by Natural Science Foundation of China (Nos: 61961003, 61561006, 11501132),  Natural Science Foundation of Guangxi (Nos: 2018JJA110110, 2016GXNSFAA380049) and  the talent project of  Education Department of Guangxi Government  for Young-Middle-Aged backbone teachers. Deguang Han  is partially supported by the NSF grant  DMS-1712602.}

\date{\today}

\begin{abstract} The Sobolev space $H^{\varsigma}(\mathbb{R}^{d})$, where $\varsigma > d/2$, is an
important function space that  has many applications in various
areas of research. Attributed to the inertia of a measurement
instrument, it is desirable in sampling theory to recover  a
function by its nonuniform sampling.
In the present paper, based on dual
framelet systems  for  the Sobolev space pair
$(H^{s}(\mathbb{R}^{d}), H^{-s}(\mathbb{R}^{d}))$, where $d/2<s<\varsigma$, we
investigate the problem of constructing the
 approximations to all the functions in $H^{\varsigma}(\mathbb{R}^{d})$ by  nonuniform sampling.
We first
establish the convergence rate  of the framelet  series in
$(H^{s}(\mathbb{R}^{d}), H^{-s}(\mathbb{R}^{d}))$, and  then
construct the framelet approximation operator that  acts  on  the entire
space $H^{\varsigma}(\mathbb{R}^{d})$.
We examine the
stability property for  the framelet
approximation operator with respect to the perturbations of shift parameters, and obtain
an estimate bound for the perturbation error. Our result  shows that
under the condition $d/2<s<\varsigma$, the approximation operator is robust
to shift perturbations. Motivated  by some recent work on nonuniform sampling and  approximation in
Sobolev space (e.g., \cite{Hamm1}), we don't require  the perturbation sequence to be in
 $\ell^{\alpha}(\mathbb{Z}^{d})$. Our results  allow us to establish the
 approximation for every function in
$H^{\varsigma}(\mathbb{R}^{d})$ by nonuniform sampling.
In particular, the approximation error is robust to the jittering of the samples.
\end{abstract}
\maketitle

\section{Introduction}\label{section1}

Sampling is a fundamental tool for the conversion between an
analogue signal and its digital form (A/D). The most classical
sampling theory is the Whittaker-Kotelnikov-Shannon (WKS) sampling
theorem \cite{C. Shannon1}, which states   that a
bandlimited signal can be perfectly reconstructed  if it is sampled
at a rate greater than its Nyquist frequency.
The WKS sampling theorem holds only for bandlimited signals. In
order to extend the sampling theorem to non-bandlimited signals,
researchers have established various sampling theorems for many
other function spaces. Such examples include  the sampling theory
for shift-invariant subspaces (c.f.\cite{Aldroubi1,Aldroubi3,Aldroubi2,Boor1,Johnson1,Johnson2,Wenchang3,df1,df0}), reproducing kernel
subspaces of $L^{2}(\mathbb{R}^{d})$ (c.f.\cite{HD2,HammJ,Hangelbroek,HD3}) and   subspaces from the
generalized sinc functions (c.f.\cite{CQL}).

\subsection{The goal of applicable scope and sampling flexibility}

For any $\varsigma\in\mathbb{R}$,  the Sobolev space $H^{\varsigma}(\mathbb{R}^{d})$
is defined as
\begin{align}\begin{array}{lllll}\label{defi_s}\displaystyle  H^{\varsigma}(\mathbb{R}^{d}):=\Big\{f: \int_{\mathbb{R}^{d}}|\widehat{f}(\xi)|^{2} (1+||\xi||_{2}^{2})^{\varsigma}d\xi<\infty\Big\},\end{array}\end{align}
 where
$\widehat{f}(\xi):=\int_{\mathbb{R}^{d}}f(x)e^{-\textbf{i}x\cdot \xi}dx$ is
the Fourier transform of $f$. For $\varsigma>d/2$, by the similar  analysis in \cite[Chapter 9.1]{Matla}
one  can check that the functions in  $H^{\varsigma}(\mathbb{R}^{d})$ are continuous. From now on it is assumed that
$\varsigma>d/2.$
The function theory in
$H^{\varsigma}(\mathbb{R}^{d})$ is important for many problems. Among others these include
the boundedness of the Fourier multiplier operator
\cite{youyong3, youyong2}, viscous shallow water system
\cite{youyong4,youyong5}, PDE or ODE \cite{HanZ,youyong6}, and signal analysis
\cite{Matla}. Moreover, it is easy to check that many
important function  spaces such as the bandlimited function space \cite{C. Shannon1},
shift-invariant  subspace  \cite{Aldroubi1,Aldroubi3,Aldroubi2,Chui,AFJ04,Hanbin1,Qiyu0}  (in which the generator is continuous)
are   contained in $H^{\varsigma}(\mathbb{R}^{d})$ for some appropriate $\varsigma>d/2$.  However,  in general it not easy to determine whether a  function in $H^{\varsigma}(\mathbb{R}^{d})$  belongs to a desired subspace or not. Therefore it is practically useful to establish some recovery methods  for the entire space
$H^{\varsigma}(\mathbb{R}^{d})$. Besides  the aspect of  \emph{applicable scope}, the samples we acquire may also well
be jittered and thus usually nonuniform \cite{Wenchangui,Wenchangui1,Qiyu0}. Therefore the goal of this paper is to establish a
sampling theory for the entire space $H^{\varsigma}(\mathbb{R^{d}}^{d})$. This will allow us  to construct the approximations to \emph{all} the functions in $H^{\varsigma}(\mathbb{R^{d}}^{d}),$ which admit   \emph{nonuniform} sampling points. To the best of our knowledge, this has not been examined in the literature. Our goal will be  achieved in Theorem \ref{Theorem x4} by  the theory of dual framelets in $(H^{s}(\mathbb{R}^{d}), H^{-s}(\mathbb{R}^{d})),$
with $d/2<s<\varsigma$,
which was introduced by Han and Shen \cite{Hanbin1}. In what follows we introduce some necessary  terminologies  for
framelets in Sobolev spaces. More details can be found in  \cite{Hanbin1} and Han's
continuing work \cite{Hanbin2,Hanbin3,BHBOOK} on dual framelets  in  distribution spaces.

\subsection{Preliminary  terminologies for  dual framelets in dual Sobolev spaces}
By \eqref{defi_s}, $H^{\varsigma}(\mathbb{R}^{d})$ is equipped with the
inner product $\langle\cdot,\cdot\rangle_{H^{\varsigma}(\mathbb{R}^{d})}$
defined by
\begin{align}\label{neiji} \begin{array}{lllll} \displaystyle \langle f,g\rangle_{H^{\varsigma}(\mathbb{R}^{d})}:=\frac{1}{(2\pi)^{d}}\int_{\mathbb{R}^{d}}\widehat{f}(\xi)\overline{\widehat{g}(\xi)}(1+||\xi||_{2}^{2})^{\varsigma}d\xi,
\ \ \ \ \forall f,g\in H^{\varsigma}(\mathbb{R}^{d}),\end{array}\end{align} where
$\overline{\widehat{g}}$ is the complex conjugate.
Naturally,  the deduced norm is defined   by
\begin{align}\notag \begin{array}{lllll} \displaystyle ||f||_{H^{\varsigma}(\mathbb{R}^{d})}:=\frac{1}{(2\pi)^{d/2}}\Big(\int_{\mathbb{R}^{d}}|\widehat{f}(\xi)|^{2}(1+||\xi||_{2}^{2})^{\varsigma}d\xi\Big)^{1/2}, \quad \forall f\in H^{\varsigma}(\mathbb{R}^{d}).\end{array}\end{align}
It is easy to check that  the  functional $\langle \cdot,
\cdot \rangle: (H^{\varsigma}(\mathbb{R}^{d}),
H^{-\varsigma}(\mathbb{R}^{d}))\longrightarrow \mathbb{C}$ defined by
\begin{align}\begin{array}{lllll}\notag \displaystyle\langle f, g
\rangle:=\frac{1}{(2\pi)^{d}}\int_{\mathbb{R}^{d}}\widehat{f}(\xi)\overline{\widehat{g}(\xi)}d\xi,
\ \forall f\in H^{\varsigma}(\mathbb{R}^{d}),  g\in H^{-\varsigma}(\mathbb{R}^{d})\end{array}\end{align}
can be bounded by  $|\langle f, g \rangle|\leq
||f||_{H^{\varsigma}(\mathbb{R}^{d})}||g||_{H^{-\varsigma}(\mathbb{R}^{d})}.$
Clearly,
$H^{\varsigma_{1}}(\mathbb{R}^{d})$ $\supseteq H^{\varsigma_{2}}(\mathbb{R}^{d})$ if
and only if $\varsigma_{1}\leq \varsigma_{2}$. Moreover,
$H^{0}(\mathbb{R}^{d})=L^{2}(\mathbb{R}^{d})$  and  correspondingly
$||\cdot||_{H^{0}(\mathbb{R}^{d})}=||\cdot||_{L^{2}}$.
 For any two functions $f, g: \mathbb{R}^{d}\longrightarrow \mathbb{C}$ and $\mu\in \mathbb{R}$, define their  bracket product $[f,
g]_{\mu}$ by
\begin{align} \begin{array}{lllll}\displaystyle\label{opq}[f,
g]_{\mu}(\xi):=\sum_{k\in\mathbb{Z}^{d}}f(\xi+2k\pi)\overline{g(\xi+2k\pi)}(1+||\xi+2k\pi||_{2}^{2})^{\mu},\end{array}\end{align}
whenever the above series converge.
Readers can  refer to Han's method \cite{WEZ,BHBOOK}
for estimating  the bracket product.

A $d\times d$ integer matrix $M$ is referred to as  a dilation
matrix if all its eigenvalues are strictly larger than $1$ in
modulus. Throughout this paper, we are interested in the case that
$M$ is isotropic. Specifically, $M$ is similar to
$\mbox{diag}(\lambda_{1},\lambda_{2},\cdots,\lambda_{d})$ with
$$|\lambda_{1}|=|\lambda_{2}|=\ldots=|\lambda_{d}|=m:=|\det M|^{1/d}.$$
%
Suppose that $\phi\in H^{s}(\mathbb{R}^{d})$ is an
$M$-refinable function given via the  $M$-refinement  equation
\begin{align}\label{yy1}\widehat{\phi}(M^{T}\cdot)=\widehat{a}(\cdot)\widehat{\phi}(\cdot),\end{align}
where $\widehat{a}(\cdot):=\sum_{k\in\mathbb{Z}^{d}}a[k]e^{ik\cdot}$
is referred to as the mask symbol of $\phi$, and
 $\{\psi^{\ell}\}^{L}_{\ell=1}$ is a set of  wavelet functions
 defined by
\begin{align}\begin{array}{lllll} \label{yy3}\widehat{\psi^{\ell}}(M^{T}\cdot)=\widehat{b^{\ell}}(\cdot)\widehat{\phi}(\cdot),\end{array}\end{align}
where the $2\pi \mathbb{Z}^{d}$-periodic trigonometric polynomial
$\widehat{b^{\ell}}(\cdot)$ is the mask symbol of $\psi^{\ell}$. Now
a wavelet system $X^{s}(\phi;$ $\psi^{1},\ldots,\psi^{L})$ in
$H^{s}(\mathbb{R}^{d})$ is defined as
\begin{align}\begin{array}{lllll}\label{lisan}
X^{s}(\phi;\psi^{1},\ldots,\psi^{L})&:=\{\phi_{0,k}:k\in
\mathbb{Z}^{d}\}
\\&\quad \cup\{\psi^{\ell,s}_{j,k}:k\in\mathbb{Z}^{d}, j\in
\mathbb{N}_{0},\ell=1,\ldots,L\},
\end{array}\end{align}
where $\phi_{0,k}=\phi(\cdot-k)$,
$\psi^{\ell,s}_{j,k}=m^{j(d/2-s)}\psi^{\ell}(M^{j}\cdot-k)$ and
$\mathbb{N}_{0}:=\mathbb{N}\cup\{0\}$. If there exist  two positive
constants $C_{1}$ and $C_{2}$ such that for every $ f\in H^{s}(\mathbb{R}^{d}),$
\begin{align}\begin{array}{lllll}\label{hanbin3}\displaystyle
C_{1}||f||_{H^{s}(\mathbb{R}^{d})}^{2}\leq
\sum_{k\in\mathbb{Z}^{d}}|\langle
f,\phi_{0,k}\rangle_{H^{s}(\mathbb{R}^{d})}|^{2}
 +\sum^{L}_{\ell=1}\sum_{j\in\mathbb{N}_{0}}\sum_{k\in
\mathbb{Z}^{d}} |\langle f,\psi^{\ell,s}_{j,k}
\rangle_{H^{s}(\mathbb{R}^{d})}|^{2} \leq
C_{2}||f||_{H^{s}(\mathbb{R}^{d})}^{2},
\end{array}\end{align}
then we say that
 $X^{s}(\phi;\psi^{1},\ldots,\psi^{L})$ is an $M$-framelet system in
$H^{s}(\mathbb{R}^{d})$. If there exists another $M$-framelet system
$X^{-s}(\widetilde{\phi};\widetilde{\psi}^{1},\ldots,\widetilde{\psi}^{L})$
in $H^{-s}(\mathbb{R}^{d})$ such that for any $ f\in
H^{s}(\mathbb{R}^{d})$ and $ g\in H^{-s}(\mathbb{R}^{d})$,
\begin{align}\begin{array}{lllll}\label{hanbin4}\displaystyle \langle f,g\rangle=\sum_{k\in \mathbb{Z}^{d}}\langle
\phi_{0,k}, g\rangle\langle f, \widetilde{\phi}_{0,k}\rangle+
\sum^{L}_{\ell=1}\sum_{j\in \mathbb{\mathbb{N}}_{0}}\sum_{k\in
\mathbb{Z}^{d}} \langle \psi^{\ell,s}_{j,k}, g \rangle\langle f,
\widetilde{\psi}^{\ell,-s}_{j,k} \rangle,
\end{array}\end{align}
then we say that $X^{s}(\phi;\psi^{1},\ldots,\psi^{L})$ and
$X^{-s}(\widetilde{\phi};\widetilde{\psi}^{1},\ldots,\widetilde{\psi}^{L})$
form a pair of dual $M$-framelet systems in
$(H^{s}(\mathbb{R}^{d}),H^{-s}(\mathbb{R}^{d}))$. For any function
$f\in H^{s}(\mathbb{R}^{d})$, it follows
 from \eqref{hanbin4} that
\begin{align}\begin{array}{lll}\label{Albert1}
\displaystyle f=\sum_{k\in \mathbb{Z}^{d}}\langle f,
\widetilde{\phi}_{0,k}\rangle \phi_{0,k}+
\sum^{L}_{\ell=1}\sum_{j\in \mathbb{\mathbb{N}}_{0}}\sum_{k\in
\mathbb{Z}^{d}} \langle f, \widetilde{\psi}^{\ell,-s}_{j,k} \rangle
\psi^{\ell,s}_{j,k}.
\end{array}\end{align}

\begin{note}\label{ghch}(i)
The framelets in $L^{2}(\mathbb{R}^{d})$ must have at least one  vanishing
moment  such that the framelet series converge  unconditionally
(c.f. \cite{AFJ04,Ehler,Hanbin1}). However
when $s>0$,
 the vanishing
moment of $\psi^{\ell}$ is not necessary for the convergence of the
series in \eqref{Albert1}, $\ell=1, \ldots, L$. This is the most
significant difference between the framelets in
$L^{2}(\mathbb{R}^{d})$ and those in
$(H^{s}(\mathbb{R}^{d}),H^{-s}(\mathbb{R}^{d}))$. For more details
about the  conditions for the convergence of  framelet series
in $(H^{s}(\mathbb{R}^{d}),H^{-s}(\mathbb{R}^{d}))$, readers can
refer  to \cite{Hanbin1,BHBOOK}. (ii) Our goal is to construct the approximations to all the functions in $H^{\varsigma}(\mathbb{R}^{d})$
($\varsigma>d/2) $. The  construction scheme is sketched as follows.  We first choose $d/2<s<\varsigma$,  and design special dual framelet systems
$X^{s}(\phi;\psi^{1},\ldots,\psi^{L})$ and
$X^{-s}(\widetilde{\phi};\widetilde{\psi}^{1},\ldots,\widetilde{\psi}^{L})$ in  $(H^{s}(\mathbb{R}^{d}),H^{-s}(\mathbb{R}^{d}))$.
 Recall that the target  $f\in H^{\varsigma}(\mathbb{R}^{d})\subseteq H^{s}(\mathbb{R}^{d})$.
 Then we shall use \eqref{Albert1} to establish the approximation to $f$. The reason for $\varsigma>s$
 is postponed to Remark \ref{smoothnessreason}. \end{note}

\subsection{Main results and structure  of the present paper}
 The main  results of  the present paper are stated in Thorem \ref{Theorem x1}, Theorem \ref{Theorem xe4} and Theorem \ref{Theorem x4}.
As assumed  in Note \ref{ghch} (ii), the target $f\in H^{\varsigma}(\mathbb{R}^{d})$,  and
$X^{s}(\phi;\psi^{1},\ldots,\psi^{L})$ and
$X^{-s}(\widetilde{\phi};\widetilde{\psi}^{1},\ldots,\widetilde{\psi}^{L})$ are the dual   framelet systems  in  $(H^{s}(\mathbb{R}^{d}),H^{-s}(\mathbb{R}^{d}))$ where
$d/2<s<\varsigma$.
It will be clear  in Theorem \ref{Theorem x4} that   the
truncation version  $\mathcal{S}_{\phi}^{N}f$ of the series in
\eqref{Albert1} with respect to the scale $j$, defined by
\begin{align}\label{rt} \mathcal{S}_{\phi}^{N}f:=\sum_{k\in
\mathbb{Z}^{d}}\langle f, \widetilde{\phi}_{0,k}\rangle
\phi_{0,k}+\sum^{L}_{\ell=1}\sum^{N-1}_{j=0}\sum_{k\in
\mathbb{Z}^{d}} \langle f, \widetilde{\psi}^{\ell,-s}_{j,k} \rangle
\psi^{\ell,s}_{j,k}, \end{align}
is crucial for establishing the sampling and  approximation.
Naturally, the first
  problem  is how to
 estimate the approximation error $||(I-\mathcal{S}_{\phi}^{N})f||_{H^{s}}$ for any $f\in H^{\varsigma}(\mathbb{R}^{d})$, where $I$ is the identity operator.
 The answer to  this problem will be given  in Theorem \ref{Theorem x1}. It should be noted that the estimation of the  approximation error established in this paper
 holds for any $f\in H^{\varsigma}(\mathbb{R}^{d})$.
 In \cite{fagep},
the error  estimation was   established   for a class of target functions.
For one-dimensional case,  the error was estimated in the sense of  Sobolev seminorm  by \cite[Corollary 4.7.3]{BHBOOK}.
Theorem \ref{Theorem x1} is  not the trivial generalization of \cite{fagep} and \cite[Corollary 4.7.3]{BHBOOK}.
More details of comparison will be given  in Comparison \ref{duibi1}.

 We next turn to the perturbation of  $\mathcal{S}_{\phi}^{N}f$.
It  will be clear  in \eqref{suanzidigyi} and \eqref{integral1}
 that the sampling  nonuniformity  is
substantially derived from the perturbation of shift parameter $k$ of the
 system $\{\widetilde{\phi}^{-s}_{N,k}\}_{k\in
\mathbb{Z}^{d}}\subseteq H^{-s}(\mathbb{R}^{d})$, where  $\widetilde{\phi}^{-s}_{N,k}=m^{N(d/2+s)}\widetilde{\phi}(M^{N}\cdot-k)$. Thus, in order to construct the  approximation by nonuniform
sampling,  we need to   estimate    the
perturbation error of $||(I-\mathcal{S}_{\phi,\underline{\varepsilon}}^{N})f||_{L^{2}}$, where $\underline{\varepsilon}=\{\varepsilon_{k}\}_{k\in \mathbb{Z}^{d}}$ is the perturbation sequence of
the  shift sequence   $\{k\}_{k\in \mathbb{Z}^{d}}$.
Theorem \ref{Theorem xe4} establishes such an error estimation.
 Motivated by Hamm's recent work \cite{Hamm1} on the nonuniform
sampling-based  approximation, the perturbation sequence in the present paper is not required to sit  in $\ell^{2}(\mathbb{Z}^{d})$(the square summable
sequence space).   Since the  perturbation sequence  is not necessarily  in $\ell^{2}(\mathbb{Z}^{d})$,
the  error can not be estimated by the
brute force estimation but by using some crucial techniques in Subsection \ref{guanjianjishu}. More details about the techniques
will be  summarized  in Subsection \ref{ganjisong}.


The paper is organized as follows. In Section \ref{sectionss}  we shall establish an error estimation  of $||(I-\mathcal{S}_{\phi}^{N})f||_{H^{s}(\mathbb{R}^{d})}$
for any $f\in H^{\varsigma}(\mathbb{R}^{d})$,  where $\varsigma>s>d/2$. The perturbation error $||f-\mathcal{S}_{\phi; \underline{\varepsilon}}^{N}f||_{L^{2}}$ will be  estimated  in Section \ref{sectionyy} (Theorem \ref{Theorem xe4}), where
$\mathcal{S}_{\phi; \underline{\varepsilon}}^{N}f$ is the perturbation version   of $\mathcal{S}_{\phi}^{N}f$ with
$\underline{\varepsilon}$ being  the  the shift-perturbation sequence mentioned previously.
In  Section \ref{bugzecy} (Theorem \ref{Theorem x4}) we will  present a main application of our two main
results in Theorem \ref{Theorem x1} and Theorem \ref{Theorem xe4}. Specifically,  by using a pair of dual framelets for
$(H^{s}(\mathbb{R}^{d}), H^{-s}(\mathbb{R}^{d})),$   we are able to
construct the nonuniform sampling-based  approximation to any function in
$H^{\varsigma}(\mathbb{R}^{d})$. Our main approximation results in  Theorem \ref{Theorem x1},  Theorem \ref{Theorem xe4}
and Theorem \ref{Theorem x4}, and the estimation techniques are not  trivial generalizations of the results available  in the literature.
In Section \ref{comparison},  we make  detailed comparisons between   the main
results and the estimation techniques  of this paper with  the existing ones in the literature.
Two simulation examples are presented in Section \ref{diwujie} to demonstrate the approximation
efficiency.


\section{Framelet approximation system in Sobolev space}\label{sectionss}
In Lemma \ref{Lemma 2.1} we will   estimate  the convergence rate of  the
coefficient sequence $\big\{\langle f,
\widetilde{\psi}^{\ell,-s}_{j,k} \rangle\big\}_{j,k}$ in \eqref{Albert1} with respect to the scale $j$.
Based on Lemma \ref{Lemma 2.1}  we will  establish an  estimation  for
$||(I-\mathcal{S}_{\phi}^{N})f||_{H^{s}(\mathbb{R}^{d})}$
in  Theorem \ref{Theorem x1}, where $f\in H^{\varsigma}(\mathbb{R}^{d})$ with $\varsigma>s$. The following notations and definitions are needed for our discussion.


For any $\alpha:=(\alpha_{1}, \alpha_{2}, \ldots,
\alpha_{d})\in \mathbb{N}^{d}_{0}$ and $x:=(x_{1}, x_{2}, \ldots,
x_{d})\in\mathbb{R}^{d}$, define
$x^{\alpha}:=\prod^{d}_{k=1}x^{\alpha_{k}}_{k}$. For a function $f:
\mathbb{R}^{d}\longrightarrow \mathbb{C}$, its $\alpha$th partial
derivative $\frac{\partial^{\alpha}}{\partial x^{\alpha}}f$ is
defined as
$$\frac{\partial^{\alpha}}{\partial x^{\alpha}}f=\frac{\partial_{1}^{\alpha_{1}}\partial_{2}^{\alpha_{2}}\cdots
\partial_{d}^{\alpha_{d}}}{\partial x_{1}^{\alpha_{1}}\partial x_{2}^{\alpha_{2}}\cdots
\partial x_{d}^{\alpha_{d}}}f.$$ We say that a function  $f:
\mathbb{R}^{d}\longrightarrow \mathbb{C}$ has $\kappa+1
(\in\mathbb{N})$ vanishing moments  if
$$\frac{\partial^{\alpha}}{\partial x^{\alpha}}\widehat{f}(0)=0$$
for any  $\alpha\in\mathbb{N}^{d}_{0}$ such that $||\alpha||_{1}\leq
\kappa$. The characteristic function of the set $E\subseteq \mathbb{R}^{d}$
is denoted by $\chi_{E}$. Motivated by \cite[Theorem 4.6.5]{BHBOOK},
we establish the convergence rate of the wavelet series in $H^{\varsigma}(\mathbb{R}^{d})$ in the following lemma.

\begin{lem}\label{Lemma 2.1}
Suppose that $\widetilde{\phi}\in H^{-\varsigma}(\mathbb{R}^{d})$ is $M$-refinable such that  $||[\widehat{\widetilde{\phi}}, \widehat{\widetilde{\phi}}]_{-\varsigma}||_{L^{\infty}(\mathbb{T}^{d})}<\infty$,
where  $\varsigma>0$ and   $\mathbb{T}=(-\pi, \pi]$.
Moreover, $\widetilde{\phi}$ belongs to $ H^{-s}(\mathbb{R}^{d})$  with $0<s<\varsigma$.
A wavelet function
$\widetilde{\psi}$ given   by
$\widehat{\widetilde{\psi}}(M^{T}\cdot)=\widehat{\widetilde{b}}(\cdot)\widehat{\widetilde{\phi}}(\cdot)$
has $\kappa+1$ vanishing moments, where $\widehat{\widetilde{b}}$ is
a $2\pi\mathbb{Z}^{d}$-periodic trigonometric polynomial,
$\kappa\in\mathbb{N}_{0}$ and $\kappa+1>\varsigma$. That is, there exists $g_{1}>0$ such that
$|\widehat{\widetilde{\psi}}(\xi)|\leq g_{1}||\xi||^{\kappa+1}_{2}$ for a.e. $\xi\in \mathbb{T}^{d}$.
 Then there exists a positive constant $H_{\widetilde{\psi}}(\varsigma,s)$ such
 that for any $ f\in H^{\varsigma}(\mathbb{R}^{d})$,
\begin{align} \label{lanlan}\sum^{\infty}_{j=N}\sum_{k\in \mathbb{Z}^{d}}|\langle
f, \widetilde{\psi}^{-s}_{j,k}\rangle|^{2}\leq
H_{\widetilde{\psi}}(\varsigma,s)||f||^{2}_{H^{\varsigma}(\mathbb{R}^{d})} m^{-2N(\varsigma-s)}. \end{align}
\end{lem}
\begin{proof}
Through  the  direct calculation we have
\begin{align}\begin{array}{lllll}\label{eq0}
\displaystyle\sum_{k\in \mathbb{Z}^{d}}|\langle f, \widetilde{\psi}^{-s}_{j,k}\rangle|^{2}\\
=\displaystyle\frac{m^{jd}}{(2\pi)^{d}}\int_{(-\pi,\pi]^{d}}m^{2js}|[\widehat{f}((M^{T})^{j}\cdot), \widehat{\widetilde{\psi}}]_{0}(\xi)|^{2}d\xi\\
\displaystyle\leq\frac{m^{jd}}{2^{d-1}\pi^{d}}\int_{(-\pi,\pi]^{d}}m^{2js}|\widehat{f}((M^{T})^{j}\xi)\widehat{\widetilde{\psi}}(\xi)|^{2}d\xi\\
\displaystyle+\frac{m^{jd}}{2^{d-1}\pi^{d}}\int_{(-\pi,\pi]^{d}}m^{2js}|\sum_{0\neq k\in \mathbb{Z}^{d}}\widehat{f}((M^{T})^{j}(\xi+2\pi k))\overline{\widehat{\widetilde{\psi}}(\xi+2\pi k)}|^{2}d\xi\\
=
I_{1,j}+I_{2,j},
\end{array}\end{align}
where the Cauchy-Schwarz inequality is used in the  inequality,
\begin{align}\begin{array}{lllll}\displaystyle\notag
I_{1,j}:=\frac{1}{2^{d-1}\pi^{d}}\int_{\mathbb{R}^{d}}|\widehat{f}(\xi)|^{2}(1+||\xi||^{2}_{2})^{\varsigma}(1+||\xi||^{2}_{2})^{-\varsigma}m^{2js}|\widehat{\widetilde{\psi}}((M^{-T})^{j}\xi)|^{2}\chi_{\Lambda_{j}}(\xi)d\xi
\end{array}\end{align}
and
\begin{align}\begin{array}{lllll}\notag
\displaystyle I_{2,j}:=\frac{||[\widehat{\widetilde{\psi}}, \widehat{\widetilde{\psi}}]_{-\varsigma}||_{L^{\infty}(\mathbb{T}^{d})}}{2^{d-1}\pi^{d}}\int_{\mathbb{R}^{d}}|\widehat{f}(\xi)|^{2}(1+||\xi||^{2}_{2})^{\varsigma}(1+||\xi||^{2}_{2})^{-\varsigma}m^{2js}(1+||(M^{T})^{-j}\xi||_{2}^{2})^{\varsigma}\chi_{\mathbb{R}^{d}\setminus\Lambda_{j}}(\xi)d\xi
\end{array}\end{align}
with $\Lambda_{j}:=(M^T)^{j}\mathbb{T}^{d}$.

Define
\begin{align}\begin{array}{lllll}\notag B_{1,N}(\xi):=(1+||\xi||^{2}_{2})^{-\varsigma}\sum^{\infty}_{j=N}m^{2js}|\widehat{\widetilde{\psi}}((M^{-T})^{j}\xi)|^{2}\chi_{\Lambda_{j}}(\xi),
\end{array}\end{align}
and
\begin{align}\begin{array}{lllll}\notag B_{2,N}(\xi):=(1+||\xi||^{2}_{2})^{-\varsigma}\sum^{\infty}_{j=N}m^{2js}(1+||(M^{T})^{-j}\xi|||_{2}^{2})^{\varsigma}\chi_{\mathbb{R}^{d}\setminus\Lambda_{j}}(\xi).
\end{array}\end{align}
It follows from \eqref{eq0} that \begin{align}\label{wuganda}\begin{array}{lllll}\displaystyle  \sum_{k\in \mathbb{Z}^{d}}|\langle f, \widetilde{\psi}^{-s}_{j,k}\rangle|^{2}\leq 2\max\{1, ||[\widehat{\widetilde{\psi}}, \widehat{\widetilde{\psi}}]_{-\varsigma}||_{L^{\infty}(\mathbb{T}^{d})}\}||f||^{2}_{H^{\varsigma}(\mathbb{R}^{d})}\max\{||B_{1,N}||_{L^{\infty}(\mathbb{R}^{d})}, ||B_{2,N}||_{L^{\infty}(\mathbb{R}^{d})}\}.\end{array}\end{align}
In what follows   we   estimate $||B_{1,N}||_{L^{\infty}(\mathbb{R}^{d})}$ and $||B_{2,N}||_{L^{\infty}(\mathbb{R}^{d})}$.
Clearly if $\xi\in \Lambda_{j}$ then $||\xi||_{2}\leq\sqrt{d}\pi m^{j}$. Consequently,
if   $0<||\xi||_{2}\leq \sqrt{d}\pi m^{N}$ then we have
\begin{align}\label{honda}\begin{array}{lllll}
\displaystyle B_{1,N}(\xi)&= (1+||\xi||^{2}_{2})^{-\varsigma}\sum^{\infty}_{j=N}m^{2js}|\widehat{\widetilde{\psi}}((M^{-T})^{j}\xi)|^{2}\chi_{\Lambda_{j}}(\xi)\\
&\leq   g^{2}_{1}(1+||\xi||^{2}_{2})^{-\varsigma}||\xi||^{2(\kappa+1)}_{2}\sum^{\infty}_{j=N}m^{-2(\kappa+1-s)j}\\
&\leq \displaystyle g^{2}_{1}(\sqrt{d}\pi )^{2(\kappa+1-\varsigma)}\frac{m^{-2(\varsigma-s) N}}{1-m^{-2(\kappa+1-s)}},
\end{array}\end{align}
where we use $\kappa+1\geq \varsigma>s$ in the last inequality.
Next we estimate $B_{1,N}(\xi)$ when $||\xi||_{2}> \sqrt{d}\pi m^{N}$.
By the above analysis, if $\xi\in \Lambda_{j}$ then $j\geq J_{\xi}:=\max\{0, \lceil\log_{m}\frac{||\xi||_{2}}{\sqrt{d}\pi}\rceil\}$,
where $\lceil x \rceil$ is the smallest integer that is larger than $x.$
Therefore, whenever   $||\xi||_{2}> \sqrt{d}\pi m^{N}$ we have
\begin{align}\begin{array}{lllll}\label{lenol} B_{1,N}(\xi)&=(1+||\xi||^{2}_{2})^{-\varsigma}\sum^{\infty}_{j=N}m^{2js}|\widehat{\widetilde{\psi}}((M^{-T})^{j}\xi)|^{2}\chi_{\Lambda_{j}}(\xi)\\
& =(1+||\xi||^{2}_{2})^{-\varsigma}\sum^{\infty}_{j=J_{\xi}}m^{2js}|\widehat{\widetilde{\psi}}((M^{-T})^{j}\xi)|^{2}\chi_{\Lambda_{j}}(\xi)\\
&\leq  g^{2}_{1}(1+||\xi||^{2}_{2})^{-\varsigma}||\xi||^{2(\kappa+1)}_{2}\sum^{\infty}_{j=J_{\xi}}m^{-2(\kappa+1-s)j}\\
&\leq  g^{2}_{1}||\xi||^{2(\kappa+1-\varsigma)}_{2}\sum^{\infty}_{j=J_{\xi}}m^{-2(\kappa+1-s)j}\\
&= g^{2}_{1}||\xi||^{-2\varsigma}_{2}||\xi||^{2(\kappa+1)}_{2}\frac{m^{-2(\kappa+1-s)J_{\xi}}}{1-m^{-2(\kappa+1-s)}}\\
&\leq g^{2}_{1}||\xi||^{-2\varsigma}_{2}||\xi||^{2(\kappa+1)}_{2}\frac{m^{-2(\kappa+1-s)\log_{m}\frac{||\xi||_{2}}{\sqrt{d}\pi}}}{1-m^{-2(\kappa+1-s)}}\\
&= \frac{g^{2}_{1}(\sqrt{d}\pi)^{2(\kappa+1-s)}}{1-m^{-2(\kappa+1-s)}}||\xi||^{-2(\varsigma-s)}_{2}\\
& \leq \frac{g^{2}_{1}(\sqrt{d}\pi)^{2(\kappa+1-s)}}{1-m^{-2(\kappa+1-s)}}(\sqrt{d}\pi m^{N})^{-2(\varsigma-s)}\\
&= g^{2}_{1}(\sqrt{d}\pi )^{2(\kappa+1-\varsigma)}\frac{m^{-2(\varsigma-s) N}}{1-m^{-2(\kappa+1-s)}},
\end{array}\end{align}
where $J_{\xi}\geq N$ is used in the first inequality.
Then  it follows from \eqref{honda} and \eqref{lenol} that  \begin{align} \label{formula0}\begin{array}{lllll} \displaystyle ||B_{1,N}||_{L^{\infty}(\mathbb{R}^{d})}\leq \frac{g^{2}_{1}(\sqrt{d}\pi )^{2(\kappa+1-\varsigma)}}{1-m^{-2(\kappa+1-s)}}m^{-2(\varsigma-s) N}.\end{array}\end{align}

We next estimate $||B_{2,N}||_{L^{\infty}(\mathbb{R}^{d})}$.
Denote the sphere $\{\xi\in \mathbb{R}^{d}: ||\xi||_{2}\leq r\}$ by $\mathbb{S}_{d}(r)$.
Clearly $(M^{T})^{j}\mathbb{T}^{d}\supseteq \mathbb{S}_{d}(m^{j}\pi)$.
As previously we can prove that $\xi\notin\mathbb{R}^{d}\setminus
(M^{T})^{j}\mathbb{T}^{d}$ if $j\geq \widehat{J}_{\xi}$, where $\widehat{J}_{\xi}:=\max\{0, \lceil\log_{m}\frac{||\xi||_{2}}{\pi}\rceil\}$.
Consequently, $B_{2,N}(\xi)=0$ when $||\xi||_{2}<\pi m^N.$ When $||\xi||_{2}\geq\pi m^N,$  $|B_{2,N}(\xi)|$
is estimated as follows,
\begin{align}\label{formula1}\begin{array}{lllll}
B_{2,N}(\xi)&=\displaystyle (1+||\xi||^{2}_{2})^{-\varsigma}\sum^{\widehat{J}_{\xi}-1}_{j=N}m^{2js}(1+||(M^{T})^{-j}\xi|||_{2}^{2})^{\varsigma}\chi_{\mathbb{R}^{d}\setminus\Lambda_{j}}(\xi)\\
&\displaystyle \leq 2^{\varsigma}(1+||\xi||^{2}_{2})^{-\varsigma}\sum^{\widehat{J}_{\xi}-1}_{j=N}m^{2js}(1+m^{-2j\varsigma}||\xi||^{2\varsigma}_{2})\\
&=\displaystyle \frac{2^{\varsigma}}{||\xi||^{2\varsigma}_{2}}\frac{m^{2Ns}m^{2s(\widehat{J}_{\xi}-N)}}{m^{2s}-1}+2^{\varsigma}\frac{||\xi||^{2\varsigma}_{2}}{(1+||\xi||^{2}_{2})^{\varsigma}}\sum^{\widehat{J}_{\xi}-1}_{j=N}m^{-2j(\varsigma-s)}\\
&\leq\displaystyle \frac{2^{\varsigma}}{||\xi||^{2\varsigma}_{2}}\frac{m^{2s\widehat{J}_{\xi}}}{m^{2s}-1}+2^{\varsigma+1}\frac{||\xi||^{2\varsigma}_{2}}{(1+||\xi||^{2}_{2})^{\varsigma}}\frac{m^{-2N(\varsigma-s)}}{1-m^{-2(\varsigma-s)}}\\
&\leq\displaystyle 2^{\varsigma}\frac{1}{||\xi||^{2\varsigma}_{2}}\frac{m^{2s\widehat{J}_{\xi}}}{m^{2s}-1}+2^{\varsigma+1}\frac{m^{-2N(\varsigma-s)}}{1-m^{-2(\varsigma-s)}}\\
&\leq\displaystyle \big(\frac{2^{\varsigma}m^{2s}}{\pi^{2\varsigma}}\frac{1}{m^{2s}-1}+\frac{2^{\varsigma+1}}{1-m^{-2(\varsigma-s)}}\big)m^{-2N(\varsigma-s)},
\end{array}\end{align}
where we use $(1+|x|^{2})^{\varsigma}\leq (2\max\{1, |x|^{2}\})^{\varsigma}\leq 2^{\varsigma}(1+|x|^{2\varsigma})$
and $\widehat{J}_{\xi}\leq 1+\log_{m}\frac{||\xi||_{2}}{\pi}$ in the first and  last inequalities, respectively.
Define
\begin{align}\notag\begin{array}{lllll} H_{\widetilde{\psi}}(\varsigma,s)&:=\displaystyle 2\max\big\{1, ||[\widehat{\widetilde{\psi}}, \widehat{\widetilde{\psi}}]_{-\varsigma}||_{L^{\infty}(\mathbb{T}^{d})}\big\}\\
&\displaystyle \ \  \times\max\big\{\frac{g^{2}_{1}(\sqrt{d}\pi )^{2(\kappa+1-\varsigma)}}{1-m^{-2(\kappa+1-s)}}, (\frac{2^{\varsigma}m^{2s}}{\pi^{2\varsigma}}\frac{1}{m^{2s}-1}+\frac{2^{\varsigma+1}}{1-m^{-2(\varsigma-s)}}\big)\big\}<\infty.\end{array}\end{align}
Now  the proof can be concluded by  \eqref{wuganda}, \eqref{formula0} and \eqref{formula1}.
\end{proof}

\begin{rem}\label{smoothnessreason}
The condition $\varsigma-s>0$ guarantees that the series in \eqref{lenol} and \eqref{formula1}
converge.
\end{rem}

Based on  Lemma \ref{Lemma 2.1}, we  estimate the approximation error
$||(I-\mathcal{S}_{\phi}^{N})f||_{H^{s}(\mathbb{R}^{d})}$ in the following theorem,
where $\mathcal{S}_{\phi}^{N}$ is defined  in \eqref{rt}.


\begin{theo}\label{Theorem x1}
Suppose that $X^{s}(\phi;\psi^{1}, \psi^{2},\ldots, \psi^{L})$ and
$X^{-s}(\widetilde{\phi};\widetilde{\psi}^{1},
 \widetilde{\psi}^{2}, \ldots, \widetilde{\psi}^{L})$ form  a pair of dual $M$-framelet systems
 for $(H^{s}(\mathbb{R}^{d}),H^{-s}(\mathbb{R}^{d}))$. Moreover,    $\phi\in
 H^{\varsigma}(\mathbb{R}^{d})$
and
 $\widetilde{\psi}^{\ell}$ has $\kappa+1$ vanishing moments, where  $0<s<\varsigma<\kappa+1$,
$\kappa\in\mathbb{N}_{0}$ and $\ell=1,2, \ldots, L$.
 Then  there exists a positive constant $C(s, \varsigma)$  such that for any $f\in
H^{\varsigma}(\mathbb{R}^{d})$,
\begin{align}\begin{array}{lll} \label{bound1}
\displaystyle
||(I-\mathcal{S}_{\phi}^{N})f||_{H^{s}(\mathbb{R}^{d})} \leq C(s,
\varsigma)||f||_{H^{\varsigma}(\mathbb{R}^{d})} m^{-N(\varsigma-s)/2}.
\end{array}
\end{align}
\end{theo}

\begin{proof}
Suppose that  $\psi^{\ell}$ and  $\widetilde{\psi}^{\ell}$ are    defined by  $\widehat{\psi^{\ell}}(M^{T}\cdot)=\widehat{b^{\ell}}(\cdot)\widehat{\phi}(\cdot)$
and $\widehat{\widetilde{\psi}^{\ell}}(M^{T}\cdot)=\widehat{\widetilde{b}^{\ell}}(\cdot)\widehat{\widetilde{\phi}}(\cdot)$, respectively.
Denote by $\ell^{2}(\mathbb{Z}^{d}\times \mathbb{N}_{0}\times
\mathbb{Z}^{d} \times L)$  the space of square summable sequences
supported on $\mathbb{Z}^{d}\times \mathbb{N}_{0}\times
\mathbb{Z}^{d} \times L.$ Let
 $\emph{\textsf{P}}: H^{s}(\mathbb{R}^{d})\rightarrow \ell^{2}(\mathbb{Z}^{d}\times \mathbb{N}_{0}\times
\mathbb{Z}^{d} \times L)$ be the analysis operator of
$X^{s}(\phi;\psi^{1}, \psi^{2},\ldots, \psi^{L})$. That is, for any
$g\in H^{s}(\mathbb{R}^{d})$ the mapping $\emph{\textsf{P}}g$ is  defined
as
\begin{align}\nonumber
\emph{\textsf{P}}g:=\Big\{\langle
g,\phi_{0,n}\rangle_{H^{s}(\mathbb{R}^{d})},\langle
g,\psi^{\ell,s}_{j,k} \rangle_{H^{s}(\mathbb{R}^{d})}:  n, k\in
\mathbb{Z}^{d}, j\in \mathbb{N}_{0},\ell=1, \ldots, L\Big\}.
\end{align}
 By \eqref{hanbin3},
$\emph{\textsf{P}}$ is a bounded operator from
$H^{s}(\mathbb{R}^{d})$ to $\ell^{2}(\mathbb{Z}^{d}\times \mathbb{N}_{0}\times
\mathbb{Z}^{d} \times L)$. Then
\begin{align}\label{fg}
||\emph{\textsf{P}}g||_{2}&\leq
||\emph{\textsf{P}}||||g||_{H^{s}(\mathbb{R}^{d})}.
\end{align}
By the isomorphic map $\theta_{s}:
H^{s}(\mathbb{R}^{d})\longrightarrow H^{-s}(\mathbb{R}^{d})$ defined
by
$$\widehat{\theta_{s}g}(\xi)=\widehat{g}(\xi)(1+||\xi||_{2}^{2})^{s}, \forall g\in H^{s}(\mathbb{R}^{d}),$$
it is easy to prove that \eqref{hanbin3} holds if and only if
\begin{align}\begin{array}{lllll}\label{xiehan}
C_{1}||\widetilde{g}||_{H^{-s}(\mathbb{R}^{d})}^{2}\leq
\sum_{k\in\mathbb{Z}^{d}}|\langle
\widetilde{g},\phi_{0,k}\rangle|^{2}
 +\sum^{L}_{\ell=1}\sum_{j\in\mathbb{N}_{0}}\sum_{k\in
\mathbb{Z}^{d}} |\langle \widetilde{g},\psi^{\ell,s}_{j,k}
\rangle|^{2} \leq
C_{2}||\widetilde{g}||_{H^{-s}(\mathbb{R}^{d})}^{2}
\end{array}\end{align}
for any $\widetilde{g}\in H^{-s}(\mathbb{R}^{d})$.
By \eqref{xiehan} and    \cite[Theorem 2.1]{Li0} we have
\begin{align}\label{fgzhuhai}
||\emph{\textsf{P}}||&\leq h(s, \varsigma),
\end{align}
where
$$h(s, \varsigma)=\big(L||[\widehat{\phi},
\widehat{\phi}]_{\varsigma}||_{L^{\infty}(\mathbb{T}^{d})}\big[1+\frac{m^{d}}{(2\pi)^{d}}(\frac{m^{2(\varsigma+s)}2^{s}}{m^{2(\varsigma-s)}-1}+\frac{2^{s}}{1-m^{-2s}})\max_{1\leq\ell\leq
L}\{||\widehat{b^{\ell}}||_{L^{\infty}(\mathbb{T}^{d})}\}\big]\big)^{1/2}.$$ Next we
compute  $\emph{\textsf{P}}^{*}$, the adjoint operator of
$\emph{\textsf{P}}$. For any $g\in
H^{s}(\mathbb{R}^{d})$ and $c\in \ell^{2}(\mathbb{Z}^{d}\times
\mathbb{N}_{0}\times \mathbb{Z}^{d} \times L)$ such that its the elements are $c_{k}$
and $c_{j,k,\ell,-s}$, we have
$$\langle \emph{\textsf{P}}^{*}c, g\rangle_{H^{s}(\mathbb{R}^{d})}=\langle c, \emph{\textsf{P}}g\rangle_{\ell^{2}}=\sum_{k\in \mathbb{Z}^{d}}c_{k}
\langle g,\phi_{0,k}\rangle_{H^{s}(\mathbb{R}^{d})}+
\sum^{L}_{\ell=1}\sum_{j\in \mathbb{\mathbb{N}}_{0}}\sum_{k\in
\mathbb{Z}^{d}}c_{j,k,\ell,-s} \langle g,\psi^{\ell,s}_{j,k}
\rangle_{H^{s}(\mathbb{R}^{d})}.$$  Therefore,
$$\emph{\textsf{P}}^{*}c=\sum_{k\in \mathbb{Z}^{d}}c_{k}
\phi_{0,k}+ \sum^{L}_{\ell=1}\sum_{j\in
\mathbb{\mathbb{N}}_{0}}\sum_{k\in \mathbb{Z}^{d}} c_{j,k,\ell,-s}
\psi^{\ell,s}_{j,k}.$$ From
$||\emph{\textsf{P}}^{*}||=||\emph{\textsf{P}}||$, we arrive at
\begin{align}\label{dfg}||\emph{\textsf{P}}^{*}(c)||_{H^{s}(\mathbb{R}^{d})}\leq
||\emph{\textsf{P}}|| ||c||_{\ell^{2}}.\end{align} For any $f\in
H^{\varsigma}(\mathbb{R}^{d})$,  it follows from \eqref{dfg},
\eqref{fgzhuhai} and \eqref{lanlan} that
\begin{align}\label{guji}\begin{array}{lll}
\displaystyle\displaystyle
||\sum^{L}_{\ell=1}\sum^{\infty}_{j=N}\sum_{k\in \mathbb{Z}^{d}}
\langle f, \widetilde{\psi}^{\ell,-s}_{j,k} \rangle
\psi^{\ell,s}_{j,k}||_{H^{s}(\mathbb{R}^{d})}
&\displaystyle\leq||\emph{\textsf{P}}||\Big(
\sum^{L}_{\ell=1}\sum^{\infty}_{j=N}\sum_{k\in \mathbb{Z}^{d}}
\displaystyle|\langle f, \widetilde{\psi}^{\ell,-s}_{j,k} \rangle|^{2}\Big)^{1/2}\\
&\leq h(s, \varsigma) \sqrt{H(\varsigma, s)} m^{-N(\varsigma-
s)/2}||f||_{H^{\varsigma}(\mathbb{R}^{d})},
\end{array}
\end{align}
where
$H(\varsigma, s):=\sum^{L}_{\ell=1}H_{\widetilde{\psi}_{\ell}}(\varsigma, s) $
with $H_{\widetilde{\psi}_{\ell}}(\varsigma, s)$ being given by Lemma  \ref{Lemma 2.1} \eqref{lanlan}.
Now we define
$$C(s, \varsigma):=h(s, \varsigma) \sqrt{H(\varsigma, s)}$$
to conclude the proof.
\end{proof}

\begin{rem}
For a target  function $f\in H^{s}(\mathbb{R}^{d})$,     the approximation error  $O(m^{-N(\varsigma-s)/2})$ in \eqref{bound1}   holds provided that
$f$ also  lies  in $H^{\varsigma}(\mathbb{R}^{d})$ for some  $\varsigma>s.$
As mentioned in Note \ref{ghch} (ii),
our aim in this paper  is to construct  the approximation to all the functions in $H^{\varsigma}(\mathbb{R}^{d})$ for any $\varsigma>d/2.$
In fact, the above requirement $\varsigma>s$ will not bring any  negative impact on our aim.
To make this point, we
sketch the procedures for constructing the approximation in Theorem \ref{Theorem x1}.
For any $f\in H^{\varsigma}(\mathbb{R}^{d})$, we choose $d/2<s<\varsigma$ and construct
a pair of dual $M$-framelet systems $X^{s}(\phi;\psi^{1}, \psi^{2},\ldots, \psi^{L})$ and
$X^{-s}(\widetilde{\phi};\widetilde{\psi}^{1},
 \widetilde{\psi}^{2}, \ldots, \widetilde{\psi}^{L})$
of  $(H^{s}(\mathbb{R}^{d}),H^{-s}(\mathbb{R}^{d}))$. Then $f$ can be approximated by $\mathcal{S}_{\phi}^{N}f$  in \eqref{bound1}.
\end{rem}

\begin{comp}\label{duibi1} (i) There are some bounds for  the error  $||(I-\mathcal{S}_{\phi}^{N})f||_{L^{2}}$ in the literature such as in \cite{Skopina,fagep}.
When $f$ belongs to the Schwartz class $\mathbf{S}(\mathbb{R}^{d})$,
the above error  was estimated  in  \cite[Theorem 16]{Skopina}.
For  the target $f$
satisfying   \begin{align} \label{tiaojian} |\widehat{f}(\xi)|\leq C(1+||\xi||_{2})^{\frac{-d-\alpha}{2}} \ \hbox{for} \ \hbox{every}\ \xi\in \mathbb{R}^{d}, \alpha>0,\end{align}
the error   was estimated in \cite{fagep}.
Clearly  there are many functions sitting in $L^{2}(\mathbb{R}^{d})\backslash\mathbf{S}(\mathbb{R}^{d})$ or in
$H^{s}(\mathbb{R}^{d})\backslash\mathbf{S}(\mathbb{R}^{d})$.
Moreover, there  are also many functions in $H^{s}(\mathbb{R}^{d})$ not satisfying \eqref{tiaojian}.
For example,
the following class  of the generalized sinc function (c.f. \cite{CQL,Kou}), given
by the Fourier transform
\begin{align} \label{tnx} \widehat{f}(\xi)=\sum_{n\in \mathbb{Z}}e^{\lambda_{n}}\chi_{[n, n+\epsilon_{n}]}(\xi),\end{align}
where $\lambda_{n}>0$ and $0<\epsilon_{n}< \min\{e^{-2\lambda_{n}}, 1\}$,  does not satisfy \eqref{tiaojian} for any $\alpha$ and $C_0$. Contrary to \cite{Skopina,fagep},  Theorem \ref{Theorem x1} holds for the entire space $H^{\varsigma}(\mathbb{R}^{d})$ with $\varsigma>0$.
Therefore, Theorem \ref{Theorem x1} is not the trivial generalization of
\cite{Skopina,fagep}. (ii) For the case of $d=1$,  the  Sobolev seminorm $|(I-\mathcal{S}_{\phi}^{N})f|_{W^{s}_{2}(\mathbb{R})}:=\big(\int_{\mathbb{R}}|\widehat{f}(\xi)-\widehat{\mathcal{S}_{\phi}^{N}f}(\xi)|^{2}|\xi|^{2s}d\xi\big)^{1/2}$
was estimated  in \cite[Corollary 4.7.3]{BHBOOK}. Clearly, $|\cdot|_{W^{s}_{2}}\leq 2\pi||\cdot||_{H^{s}}$
but the two norms are not equivalent. Therefore for $d=1$,  Theorem \ref{Theorem x1} is not the trivial  generalization of
\cite[Corollary 4.7.3]{BHBOOK}.
\end{comp}

\section{Approximation by the  shift-perturbed  system in $H^{s}(\mathbb{R}^{d})$ when the perturbation sequence  is not necessarily in $\ell^{\alpha}(\mathbb{Z}^{d})$}
\label{sectionyy}

The complete set of representatives of distinctive
cosets of the quotient group
$[(M^{T})^{-1}\mathbb{Z}^{d}]/\mathbb{Z}^{d}$
is denoted by
$\Gamma_{M^{T}}:=\{\gamma_{0}, \ldots, \gamma_{m^{d}-1}\}$ with $\gamma_{0}=0$.
Recall that the mixed extension principle (MEP) is an efficient algorithm  (c.f.\cite{Ehler,Ehler1,Hanbin1}) for designing  dual framelet  systems.
In this section, we will use MEP to design   the dual framelet systems
$X^{s}(\phi;\psi^{1}, \psi^{2},\ldots, \psi^{L})$ and
$X^{-s}(\widetilde{\phi};\widetilde{\psi}^{1},
 \widetilde{\psi}^{2}, \ldots, \widetilde{\psi}^{L})$ where $s>d/2$. By such systems, we will construct the  approximations to the functions in $H^{\varsigma}(\mathbb{R}^{d})$
where $\varsigma>s>d/2$.
 Since the dual systems are derived from   MEP,
  the mask symbols
 $\{ \widehat{b^{1}},
\ldots, \widehat{b^{L}}\}$ of $\{\psi^{1}, \ldots,
\psi^{L}\}$, and $\{ \widehat{\widetilde{b}^{1}}, \ldots,
\widehat{\widetilde{b}^{L}}\}$ of $\{\widetilde{\psi}^{1},
\ldots, \widetilde{\psi}^{L}\}$  satisfy
\begin{align}\label{hanbin9} \left\{\begin{array}{llllll} \sum^{L}_{\ell=1}\overline{\widehat{b^{\ell}}}\widehat{\widetilde{b}^{\ell}}+\overline{\widehat{a}}\widehat{\widetilde{a}}=1,\\
\sum^{L}_{\ell=1}\overline{\widehat{b^{\ell}}}(\cdot+\gamma_{j})\widehat{\widetilde{b}^{\ell}}(\cdot)+\overline{\widehat{a}}(\cdot+\gamma_{j})\widehat{\widetilde{a}}(\cdot+\gamma_{j})=0, \ \forall  j\in\{1, \ldots, m^{d}-1\},\end{array}\right.\end{align}
where $\widehat{a}$ and  $\widehat{\widetilde{a}}$ are  the mask symbols of
$\phi$ and $\widetilde{\phi}$, respectively. It follows from  \eqref{hanbin9} that
\begin{align}\label{rt1} \begin{array}{llllll}  \mathcal{S}_{\phi}^{N}f=\sum_{k\in \mathbb{Z}^{d}} \langle
f,\widetilde{\phi}^{-s}_{N,k} \rangle
\phi^{s}_{N,k},\end{array}\end{align} where
\begin{align}\label{dingyi}\begin{array}{llllll}
\phi^{s}_{N,k}:=m^{N(d/2-s)}\phi(M^{N}\cdot-k)  \ \hbox{and} \
\widetilde{\phi}^{-s}_{N,k}:=m^{N(d/2+s)}\widetilde{\phi}(M^{N}\cdot-k).\end{array}\end{align}
 That is, $\mathcal{S}_{\phi}^{N}$ can be reexpressed by
the system $\{\widetilde{\phi}^{-s}_{N,k},
\phi^{s}_{N,k}\}_{k\in \mathbb{Z}^{d}}$. By
\eqref{bound1},  when the scale  $N$ is sufficiently large, $f$
can be well approximated  by using the inner products $
\langle f, \widetilde{\phi}^{-s}_{N,k} \rangle, k\in\mathbb{Z}^{d}$.
In what follows we introduce the perturbed version
of $\mathcal{S}_{\phi}^{N}f.$ Motivated by Hamm \cite{Hamm1},
suppose that the perturbation   sequence
$\underline{\varepsilon}:=\{\varepsilon_{k}: k\in \mathbb{Z}^{d}\}\subseteq \mathbb{R}^{d}$  satisfies
\begin{align}\label{raodongxulie}\begin{array}{lllll}
\big(\sum_{k\in \mathbb{Z}^{d}}||\varepsilon_{k}-\lambda||_{2}^{\alpha}\big)^{1/\alpha}<\infty
\end{array}
\end{align}
for some  $\lambda\in \mathbb{R}^{d}$ and $\alpha>0$. Clearly, if $\lambda\neq0$ then  $\underline{\varepsilon}$ does  not lie  in $\ell^{\alpha}(\mathbb{Z}^{d})$  defined by
\begin{align}\label{fanshu}\begin{array}{lllll}
\ell^{\alpha}(\mathbb{Z}^{d}):=\Big\{\{x_{k}\}_{k\in \mathbb{Z}^{d}}\subseteq\mathbb{R}^{d}: ||\{x_{k}\}||_{\ell^{\alpha}(\mathbb{Z}^{d})}=(\sum_{k\in \mathbb{Z}^{d}}||x_{k}||_{2}^{\alpha})^{1/\alpha}<\infty\Big\}.
\end{array}
\end{align}
Now define
the perturbed version  $\mathcal{S}_{\phi, \underline{\varepsilon}}^{N}$ of $\mathcal{S}_{\phi}^{N}$   $: H^{s}(\mathbb{R}^{d}) \longrightarrow
L^{2}(\mathbb{R}^{d})$    by
\begin{align} \label{suanzidigyi}  \begin{array}{lllll} \mathcal{S}_{\phi; \underline{\varepsilon}}^{N}f= \sum_{k\in \mathbb{Z}^{d}} \langle f,
m^{N(d/2+s)}\widetilde{\phi}(M^{N}\cdot-k-\varepsilon_{k})\rangle
\phi^{s}_{N,k}, \forall f\in H^{s}(\mathbb{R}^{d}).\end{array}\end{align}
Our main task    in the present  section  is to establish the
approximation error  $||(I-\mathcal{S}_{\phi; \underline{\varepsilon}}^{N})f||_{L^{2}}$ for any   function $f\in H^{\varsigma}(\mathbb{R}^{d})$ with
$\varsigma>s>d/2$.
The main result    is stated  in the following theorem.

\begin{theo}\label{Theorem xe4} Suppose that  $\varsigma>d/2$ and $\phi\in H^{\varsigma}(\mathbb{R}^{d})$ is $M$-refinable. Choose
$d/2< s<\varsigma$ and construct an $M$-refinable function $\widetilde{\phi}\in H^{-s}(\mathbb{R}^{d})$.
Moreover,
suppose that  the  sequence
$\underline{\varepsilon}=\{\varepsilon_{k}\}_{k\in\mathbb{Z}^{d}}$ satisfies  in \eqref{raodongxulie} for some  $\lambda\in \mathbb{R}^{d}$ and $0<\alpha<\min\{2s-d, 2\}$, and $N\geq \frac{2s+2-\alpha}{2-\alpha}\log_{m}d$ is arbitrary.
 Then there exists a constant  $C_{3}(s,\varsigma,\alpha,d)>0$  such
 that for any  $ f\in H^{\varsigma}(\mathbb{R}^{d}),$
\begin{align} \begin{array}{lllll} \label{jyy}
 ||(I-\mathcal{S}_{\phi;\underline{\varepsilon}}^{N})f||_{L^{2}}\leq||(I-\mathcal{S}_{\phi}^{N})f||_{L^{2}}
+
C_{3}(s,\varsigma,d,\alpha)||f||_{H^{\varsigma}(\mathbb{R}^{d})}m^{-N\zeta}\big[||\lambda||_{2}^{\zeta}+
||\underline{\varepsilon}-\lambda||_{\max} \big],
\end{array}\end{align}
where \begin{align}\label{constant0}
\zeta=\min\Big\{1,\varsigma-s,
\Big(\frac{4s+(\alpha-2)d}{2s-\alpha+2}+d\Big)/2 \Big\},\end{align}  and
$||\underline{\varepsilon}-\lambda||_{\max}=\max\{||\{\varepsilon_{k}-\lambda\}_{k\in \mathbb{Z}^{d}}||_{\ell^{2}(\mathbb{Z}^{d})},
||\{\varepsilon_{k}-\lambda\}_{k\in \mathbb{Z}^{d}}||^{\alpha/2}_{\ell^{\alpha}(\mathbb{Z}^{d})}\}$.
\end{theo}

\begin{proof}
The proof is given in Subsection \ref{proofmianresult}.
\end{proof}


\subsection{Auxiliary results for proving Theorem \ref{Theorem xe4}}\label{guanjianjishu}
In this subsection we present some auxiliary results which will be  helpful for proving   Theorem \ref{Theorem xe4}.

\begin{lem}\label{Lemma Xr} Let $J\geq\log_{m}d$ and $s>d/2$. Then
\begin{align} \label{op3} \begin{array}{lllll}\displaystyle \sum_{j\in\mathbb{Z}^d, ||j||_{2}\geq m^{J}}||j||_{2}^{-2s}\leq
\widehat{C}(s,d)m^{-J(2s-d)},\end{array}\end{align}
where $$\widehat{C}(s,d):=2^{d-1}d^{s-d}\Big[\big(\frac{1}{2s-d}+1\big)\sum^{d-1}_{n=1}\tbinom{d-1}{n}\prod^{n}_{l=1}\frac{1}{2s-l}+\frac{1}{2s-1}+1\Big].$$
\end{lem}

\begin{proof}
The proof is given in the Appendix section.
\end{proof}

As mentioned previously the   perturbation sequence  $\underline{\varepsilon}$ satisfying  \eqref{raodongxulie}
does not necessarily  sit   in $\ell^{\alpha}(\mathbb{Z}^{d})$. When it sits   in $\ell^{\alpha}(\mathbb{Z}^{d})$, we establish the
approximation error $||(I-\mathcal{S}_{\phi; \underline{\varepsilon}}^{N})f||_{L^{2}}$ in the following lemma,
which will be used in the proof of   Theorem \ref{Theorem xe4}.

\begin{lem}\label{Theorem x2}
Suppose that both  $\phi\in H^{s}(\mathbb{R}^{d})$ and $\widetilde{\phi}\in
H^{-s}(\mathbb{R}^{d})$  are  $M$-refinable, where $s>d/2$  and
$||\widehat{\widetilde{\phi}}||_{L^{\infty}(\mathbb{R}^{d})}<\infty$.
Moreover, suppose that  $\underline{\varepsilon}:=\{\varepsilon_{k}\}_{k\in\mathbb{Z}^{d}}$ lies in
$\ell^{\alpha}(\mathbb{Z}^{d})$
with  $0<\alpha<\min\{2s-d, 2\}$. Then there exists a positive
constant $C_{2}(s,\alpha,d)$ such that
for any $f\in H^{s}(\mathbb{R}^{d})$ and $N\geq
\frac{2s+2-\alpha}{2-\alpha}\log_{m}d$,
\begin{align}\label{raodong1}\begin{array}{lllll}   ||(I-\mathcal{S}_{\phi; \underline{\varepsilon}}^{N})f||_{L^{2}}
\leq ||(I-\mathcal{S}_{\phi}^{N})f||_{L^{2}}+C_{2}(s,
\alpha,d)||f||_{H^{s}(\mathbb{R}^{d})}||\underline{\varepsilon}||_{\max}
m^{-N(\frac{4s+(\alpha-2)d}{2s-\alpha+2}+d)/2},
\end{array}\end{align}
where
\begin{align}\label{fszd}\begin{array}{lllll}  ||\underline{\varepsilon}||_{\max}:=\max\big\{||\underline{\varepsilon}||_{\ell^{2}(\mathbb{Z}^{d})},
||\underline{\varepsilon}||^{\alpha/2}_{\ell^{\alpha}(\mathbb{Z}^{d})}\big\}.\end{array}\end{align}
\end{lem}

\begin{proof}
We only need to prove that
\begin{align}\label{xiangdui05}\begin{array}{lllll}  ||(\mathcal{S}_{\phi; \underline{\varepsilon}}^{N}-\mathcal{S}_{\phi}^{N})f||_{L^{2}}\leq C_{2}(s,
\alpha,d)||f||_{H^{s}(\mathbb{R}^{d})}||\underline{\varepsilon}||_{\max}
m^{-N(\frac{4s+(\alpha-2)d}{2s-\alpha+2}+d)/2}.\end{array}\end{align}
The above  inequality will be proved in
subsection \ref{proof of lemma3.3}.
\end{proof}

\begin{lem}\label{Theorem x3}Suppose that
both   $\phi\in
H^{\varsigma}(\mathbb{R}^{d})$ and  $\widetilde{\phi}\in H^{-s}(\mathbb{R}^{d})$ are  $M$-refinable, where  $\varsigma>s>d/2$.
Moreover,   a  sequence
$\underline{\eta} \in  \ell^{\alpha}(\mathbb{Z}^{d})$  with
$0<\alpha<\min\{2s-d, 2\}$.
Then there exists a constant  $\widetilde{C}_{2}(s,\varsigma,\alpha,d)
>0$  such that for any  $N\geq \frac{2s+2-\alpha}{2-\alpha}\log_{m}d$, $\lambda \in \mathbb{R}^{d}$  and every   $ f\in
H^{\varsigma}(\mathbb{R}^{d})$,
\begin{align}
\begin{array}{lllll} \label{ghgvb}
||\big(I-\mathcal{S}_{\phi;\underline{\eta}}^{N}\big)(f-f(\cdot+M^{-N}\lambda))||_{L^2}
\leq
\widetilde{C}_{2}(s,\varsigma,\alpha,d)||f||_{H^{\varsigma}(\mathbb{R}^{d})}m^{-N\vartheta(s,\varsigma,\alpha,\zeta)}||\lambda||_{2}^{\zeta},
\end{array}\end{align}
where
$\vartheta(s,\varsigma,\alpha,\zeta):=\zeta+\min\big\{(\varsigma-s)/2,  \big(\frac{4s+(\alpha-2)d}{2s-\alpha+2}+d\big)/2\big\}$  with
$\zeta$   defined in \eqref{constant0}.
\end{lem}



\begin{proof}
We first  estimate
$||\big(I-\mathcal{S}_{\phi;\underline{\eta}}^{N}\big)\big(f-f(\cdot+M^{-N}\lambda)\big)||_{L^2}$
as follows,
\begin{align} \begin{array}{lllll} \label{raodong} \displaystyle ||\big(I-\mathcal{S}_{\phi;\underline{\eta}}^{N}\big)\big(f-f(\cdot+M^{-N}\lambda)\big)||_{L^2}\\
\displaystyle \leq
 ||\big(I-\mathcal{S}_{\phi}^{N}\big)\big(f-f(\cdot+M^{-N}\lambda)\big)||_{L^2}+||\big(\mathcal{S}_{\phi;\underline{\eta}}^{N}-\mathcal{S}_{\phi}^{N}\big)\big(f-f(\cdot+M^{-N}\lambda)\big)||_{L^2}\\
\displaystyle \leq
 ||\big(I-\mathcal{S}_{\phi}^{N}\big)\big(f-f(\cdot+M^{-N}\lambda)\big)||_{L^2}\\
 \ \displaystyle +C_{2}(s,
\alpha,d) ||\underline{\eta}||_{\max}m^{-N(\frac{4s+(\alpha-2)d}{2s-\alpha+2}+d)/2}||f-f(\cdot+M^{-N}\lambda)||_{H^{s}(\mathbb{R}^{d})},
\end{array}\end{align}
where first and second inequalities are derived from the triangle inequality and Lemma  \ref{Theorem x2} \eqref{xiangdui05}, respectively.

Invoking  \eqref{bound1}, we get
\begin{align}\label{zjxc} \begin{array}{lllll} ||\big(I-\mathcal{S}_{\phi}^{N}\big)\big(f-f(\cdot+M^{-N}\lambda)\big)||_{L^2}\leq C(s,
\varsigma) m^{-N(\varsigma-s)/2}||f-f(\cdot+M^{-N}\lambda)||_{H^{s}(\mathbb{R}^{d})}. \end{array}\end{align}
 On the other hand,
\begin{align}\label{dddf}\begin{array}{lllll}
 \displaystyle||f-f(\cdot+M^{-N}\lambda)||_{H^{s}(\mathbb{R}^{d})}\\
= \displaystyle\Big[\frac{1}{(2\pi)^{d}}\int_{\mathbb{R}^{d}}|\widehat{f}(\xi)(1-e^{i(M^{T})^{-N}\lambda\cdot\xi})|^{2}
(1+||\xi||^{2}_{2})^{s}d\xi\Big]^{1/2}\\
 \displaystyle\leq\Big[\frac{1}{(2\pi)^{d}}\int_{\mathbb{R}^{d}}4|\sin\big((M^{T})^{-N}\lambda\cdot\xi/2\big)|^{2\zeta}
 \displaystyle|\widehat{f}(\xi)|^{2}(1+||\xi||^{2}_{2})^{s}d\xi\Big]^{1/2}\\
 \displaystyle\leq\Big[\frac{2^{2-2\zeta}||(M^{T})^{-N}\lambda||^{2\zeta}_{2}}{(2\pi)^{d}}\int_{\mathbb{R}^{d}}||\xi||_{2}^{2\zeta}
|\widehat{f}(\xi)|^{2}(1+||\xi||^{2}_{2})^{s}d\xi\Big]^{1/2}\\
\leq\displaystyle2^{1-\zeta}m^{-N\zeta}||\lambda||_{2}^{\zeta}||f||_{H^{\varsigma}(\mathbb{R}^{d})},
\end{array}\end{align}
where the first and third   inequalities  are  derived from $\zeta\leq1$ and $\zeta\leq \varsigma-s,$
respectively.
Now by \eqref{raodong}, \eqref{zjxc} and \eqref{dddf} we have
\begin{align}\label{yuuu}\begin{array}{lllll} \displaystyle ||\big(I-\mathcal{S}_{\phi;\underline{\eta}}^{N}\big)(f-f(\cdot+M^{-N}\lambda))||_{L^2}\\
\displaystyle\leq\Big(C(s,
\varsigma) m^{-N(\varsigma-s)/2}+C_{2}(s,
\alpha,d) ||\underline{\eta}||_{\max}m^{-N(\frac{4s+(\alpha-2)d}{2s-\alpha+2}+d)/2}\Big)2^{1-\zeta}m^{-N\zeta}||\lambda||_{2}^{\zeta}||f||_{H^{\varsigma}(\mathbb{R}^{d})}\\
\leq 2^{1-\zeta}\max\big\{C(s,
\varsigma),
C_{2}(s, \alpha,d) \big\}||f||_{H^{\varsigma}(\mathbb{R}^{d})}m^{-N\vartheta(s,\varsigma,\alpha,\zeta)}||\lambda||_{2}^{\zeta}.\end{array}\end{align}
Define
\begin{align}\nonumber\begin{array}{lllll}
\widetilde{C}_{2}(s,\varsigma,\alpha,d):=2^{1-\zeta}\max\big\{C(s,
\varsigma),
\ \ C_{2}(s, \alpha,d) \big\}
\end{array}\end{align}
to conclude the proof of \eqref{ghgvb}.
\end{proof}

\subsection{Proof of Theorem \ref{Theorem xe4}}\label{proofmianresult}

By Parseval identity, we have
\begin{align}\begin{array}{lllll} \label{split1}
\displaystyle\big\langle f,
m^{Nd/2}\widetilde{\phi}(M^{N}\cdot-k-\varepsilon_{k}+\lambda)
-m^{Nd/2}\widetilde{\phi}(M^{N}\cdot-k-\varepsilon_{k})\big\rangle\\
=\displaystyle\frac{m^{-Nd/2}}{(2\pi)^{d}}\displaystyle\int_{\mathbb{R}^{d}}\widehat{f}(\xi)(1-e^{i(M^{T})^{-N}\lambda\cdot\xi})\overline{\widehat{\widetilde{\phi}}((M^{T})^{-N}\xi)}e^{i(M^{T})^{-N}(k+\varepsilon_{k}-\lambda)\xi}d\xi\\
=\displaystyle\big\langle f-f(\cdot+M^{-N}\lambda),
m^{Nd/2}\widetilde{\phi}(M^{N}\cdot-k-\varepsilon_{k}+\lambda)\big\rangle.\end{array}
\end{align}
Using \eqref{split1} and the triangle inequality, the error $
||(I-\mathcal{S}_{\phi;\underline{\varepsilon}}^{N})f||_{2}$ is estimated as
follows,
\begin{align} \begin{array}{lllll} \label{raodong011}
 ||(I-\mathcal{S}_{\phi;\underline{\varepsilon}}^{N})f||_{L^2}\\
\leq ||(I-\mathcal{S}_{\phi;\underline{\varepsilon}-\lambda}^{N})f||_{L^2}+
||\mathcal{S}_{\phi;\underline{\varepsilon}-\lambda}^{N}(f-f(\cdot+M^{-N}\lambda))||_{L^2}\\
\leq
||(I-\mathcal{S}_{\phi;\underline{\varepsilon}-\lambda}^{N})f||_{L^2}+
||(I-\mathcal{S}_{\phi;\underline{\varepsilon}-\lambda}^{N})(f-f(\cdot+M^{-N}\lambda))||_{L^2}+||f-f(\cdot+M^{-N}\lambda)||_{L^2}\\
=I_{1}+I_{2},
\end{array}\end{align}
where $$I_{1}=||(I-\mathcal{S}_{\phi;\underline{\varepsilon}-\lambda}^{N})f||_{L^2}, I_{2}=||(I-\mathcal{S}_{\phi;\underline{\varepsilon}-\lambda}^{N})(f-f(\cdot+M^{-N}\lambda))||_{L^2}+||f-f(\cdot+M^{-N}\lambda)||_{L^2}.$$
Since $\underline{\varepsilon}-\lambda \in  \ell^{\alpha}(\mathbb{Z}^{d})$, it follows from Lemma \ref{Theorem x3} \eqref{ghgvb} and
\eqref{dddf} that
\begin{align}
\begin{array}{lllll} \label{gbgb}
I_{2}&\displaystyle\leq\widetilde{C}_{2}(s,\varsigma,\alpha,d)||f||_{H^{\varsigma}(\mathbb{R}^{d})}m^{-N\vartheta(s,\varsigma,\alpha,\zeta)}||\lambda||_{2}^{\zeta}+2^{1-\zeta}||\lambda||_{2}^{\zeta}m^{-N\zeta}||f||_{H^{\varsigma}(\mathbb{R}^{d})}\\
&=\displaystyle(\widetilde{C}_{2}(s,\varsigma,\alpha,d)+2^{1-\zeta})||\lambda||_{2}^{\zeta}m^{-N\zeta}||f||_{H^{\varsigma}(\mathbb{R}^{d})}.
\end{array}\end{align}
By Lemma  \ref{Theorem x2} \eqref{raodong1} and  $\zeta\leq(\frac{4s+(\alpha-2)d}{2s-\alpha+2}+d)/2$, we have
\begin{align} \begin{array}{lllll} \label{huaz}
I_{1}\leq
||(I-\mathcal{S}_{\phi}^{N})f||_{L^2}+C_{2}(s,
\alpha,d)
||\underline{\varepsilon}-\lambda||_{\max}m^{-N\zeta}||f||_{H^{s}(\mathbb{R}^{d})}.
\end{array}\end{align}
Now it follows from  \eqref{raodong011},  \eqref{gbgb} and \eqref{huaz} that
  \begin{align} \begin{array}{lllll} \label{zhgj}
 ||(I-\mathcal{S}_{\phi;\underline{\varepsilon}}^{N})f||_{L^2}\\
\leq I_{1}+I_{2}\\
\leq \displaystyle||(I-\mathcal{S}_{\phi}^{N})f||_{L^2}+[(\widetilde{C}_{2}(s,\varsigma,\alpha,d)+2^{1-\zeta})||\lambda||_{2}^{\zeta}
+C_{2}(s,
\alpha,d)
||\underline{\varepsilon}-\lambda||_{\max}]m^{-N\zeta}||f||_{H^{s}(\mathbb{R}^{d})}.
\end{array}\end{align}
Define
$$C_{3}(s,\varsigma,\alpha,d):=\max\big\{(\widetilde{C}_{2}(s,\varsigma,\alpha,d)+2^{1-\zeta}),  C_{2}(s,
\alpha,d)\big\}$$ to conclude the proof.

%
%
%

\section{Approximations to functions in Sobolev spaces by nonuniform sampling}
\label{bugzecy}
This section starts with the definition of the  sum rule of a refinable function.
Let the $M$-refinable function $\phi\in H^{s}(\mathbb{R}^{d})$ be defined via the $M$-refinement  equation:
$\widehat{\phi}(M^{T}\cdot)=\widehat{a}(\cdot)\widehat{\phi}(\cdot).$
We say that $\phi$ has $\kappa+1$ sum rules if $\widehat{a}(\xi+\gamma_{j})=O(||\xi||_{2})$ as $\xi\rightarrow 0$,
where any  $\gamma_{j}\in [(M^{T})^{-1}\mathbb{Z}^{d}]/\mathbb{Z}^{d}$ with $j\neq0$ is  as  in \eqref{hanbin9}.
For the relationship  between the sum rule of $\phi$ and the approximation order of the shift-invariant space generated  from
$\phi$, readers can refer to \cite{BHBOOK}.

With the help of   Theorem  \ref{Theorem
xe4} we establish the  approximation in the following theorem, which
states that  any function in $H^{\varsigma}(\mathbb{R}^{d})$ (where $\varsigma>d/2$)
can be stably reconstructed by its  nonuniform sampling.

\begin{theo}\label{Theorem x4}  Suppose that     $\phi\in H^{\varsigma}(\mathbb{R}^{d})$ is   $M$-refinable and has $\kappa+1$
 sum rules where $d/2<\varsigma<\kappa+1$. Moreover,
the perturbation sequence  $\underline{\varepsilon}$ is as in \eqref{raodongxulie} for some  $\lambda\in \mathbb{R}^{d}$ and $0<\alpha<\min\{2s-d, 2\}$ where  $d/2<s<\varsigma$,
and $N\geq \frac{2s+2-\alpha}{2-\alpha}\log_{m}d$ is
arbitrary.
 Then  there exists a positive constant
 $C_{0}(s, \varsigma,\alpha,d)$  such that for any $ f\in H^{\varsigma}(\mathbb{R}^{d}),$
\begin{align}\begin{array}{lll} \label{irsamapp}
\|f-\sum_{k\in\mathbb{Z}^{d}}f(M^{-N}(k+\varepsilon_{k}))\phi(M^{N}\cdot-k)\|_{L^2}\\
 \leq
C_{0}(s,\varsigma,\alpha,d)||f||_{H^{\varsigma}(\mathbb{R}^{d})}\big[
m^{-(\varsigma-s)
N}+m^{-N\zeta}\big(||\lambda||_{2}^{\zeta}+
||\underline{\varepsilon}-\lambda||_{\max}\big) \big],
\end{array}
\end{align}
where $
\zeta$ is  defined in \eqref{constant0}, and as in Theorem \ref{Theorem xe4},
$||\underline{\varepsilon}-\lambda||_{\max}=\max\{||\{\varepsilon_{k}-\lambda\}_{k\in \mathbb{Z}^{d}}||_{\ell^{2}(\mathbb{Z}^{d})},
||\{\varepsilon_{k}-\lambda\}_{k\in \mathbb{Z}^{d}}||^{\alpha/2}_{\ell^{\alpha}(\mathbb{Z}^{d})}\}$.
\end{theo}

\begin{proof}
Define  a distribution $\Delta$ on $\mathbb{R}^{d}$ by
\begin{align}\label{caiyanghanshu}\begin{array}{lll} \Delta(x_{1}, x_{2},
\ldots, x_{d}):=\delta(x_{1})\otimes\delta(x_{2})\otimes\cdots
\otimes\delta(x_{d}),\end{array}\end{align} where $\delta$ is the delta
distribution on $\mathbb{R}$, and $\otimes$ is the tensor product. It
follows from $\widehat{\delta}\equiv1$ that $\Delta\in
H^{-\mu}(\mathbb{R}^{d})$ is $M$-refinable for any $\mu>d/2$.
Since $\phi$  has $\kappa+1$ sum rules, by MEP \cite[Algorithm
4.1]{Li0}  we can design  a pair of dual $M$-framelet systems
$X^{s}(\phi;\psi^{1}, \psi^{2},\ldots, \psi^{m^{d}})$ and
$X^{-s}(\Delta;\widetilde{\psi}^{1},
 \widetilde{\psi}^{2}, \ldots, \widetilde{\psi}^{m^{d}})$ such that $\widetilde{\psi}^{1},
 \widetilde{\psi}^{2}, \ldots,$ and $\widetilde{\psi}^{m^{d}}$ have $\kappa+1$ vanishing
moments. By  the sampling property of $\delta$, we have
\begin{align}\label{integral} \langle f, \Delta\rangle=f(0).\end{align}
Combining  \eqref{rt1} and \eqref{integral}, the operators
$\mathcal{S}_{\phi}^{N}$ and $\mathcal{S}_{\phi;\underline{\varepsilon}}^{N}$
defined in \eqref{rt} and \eqref{suanzidigyi} can be expressed by
\begin{align}\label{integral1}\begin{array}{lll} \mathcal{S}_{\phi}^{N}f=\sum_{k\in\mathbb{Z}^{d}}f(M^{-N}k)\phi(M^{N}\cdot-k), \ \mathcal{S}_{\phi;\underline{\varepsilon}}^{N}f=\sum_{k\in\mathbb{Z}^{d}}f(M^{-N}(k+\varepsilon_{k}))\phi(M^{N}\cdot-k).\end{array}\end{align}
It follows from  Theorem \ref{Theorem x1} \eqref{bound1} and Theorem \ref{Theorem xe4} \eqref{jyy} that
\begin{align}\begin{array}{lll} \nonumber
||f-\sum_{k\in\mathbb{Z}^{d}}f(M^{-N}(k+\varepsilon_{k}))\phi(M^{N}\cdot-k)||_{L^2}\\
\leq ||f||_{H^{\varsigma}(\mathbb{R}^{d})}\big[ C(s, \varsigma)
m^{-(\varsigma-s)N}+C_{3}(s,\varsigma,\alpha,d)||f||_{H^{\varsigma}(\mathbb{R}^{d})}m^{-N\zeta}\big(||\lambda||_{2}^{\zeta}+
||\underline{\varepsilon}-\lambda||_{\max} \big)\\
\leq C_{0}(s,\varsigma,\alpha,d)||f||_{H^{\varsigma}(\mathbb{R}^{d})}\big[
m^{-(\varsigma-s)N}+m^{-N\zeta}\big(||\lambda||_{2}^{\zeta}+
||\underline{\varepsilon}-\lambda||_{\max} \big)\big],
\end{array}
\end{align}
where $C_{0}(s,\varsigma,\alpha,d):=\max\{C(s, \varsigma), C_{3}(s,\varsigma,\alpha,d)\}.$
\end{proof}


\begin{rem}\label{nonu}
Suppose that the perturbation sequence  $\underline{\varepsilon}^{(N)}=
\{\underline{\varepsilon}_{k}^{(N)}\}_{k\in \mathbb{Z}^{d}}$ at the  scale $N$
satisfies $\big(\sum_{k\in \mathbb{Z}^{d}}||\underline{\varepsilon}_{k}^{(N)}-\lambda^{(N)}||_{2}^{\alpha}\big)^{1/\alpha}<\infty$
for $\lambda^{(N)}\in \mathbb{R}^d.$
If
\begin{align}\label{tj}\begin{array}{lll}
||\lambda^{(N)}||_{2}^{\zeta}+
||\underline{\varepsilon}^{(N)}-\lambda^{(N)}||_{\max}
=\hbox{o}(m^{N\zeta}),\end{array}\end{align} then it follows from Theorem \ref{Theorem x4} \eqref{irsamapp} that
$\lim_{N\rightarrow\infty}\sum_{k\in\mathbb{Z}^{d}}f(M^{-N}(k+\varepsilon_{k}))\phi(M^{N}\cdot-k)=f$.
\end{rem}


\section{Comparison with  prior work: the scope of application, sampling flexibility and estimation techniques}\label{comparison}
The sampling-based  approximation in Theorem \ref{Theorem x4} \eqref{irsamapp}
enjoys  the two  required properties:  (i) scope of application:
the entire space  $H^{\varsigma}(\mathbb{R}^{d})$ with $\varsigma>d/2$; (ii)  flexibility of sampling:
the approximation is conducted   by  the nonuniform sampling  $\{f(M^{-N}(k+\varepsilon_{k}))\}$,
where the sequence $\underline{\varepsilon}=\{\varepsilon_{k}\}$ just need to satisfy \eqref{raodongxulie} and is not necessary in $\ell^{\alpha}(\mathbb{Z}^d)$.

There are some papers addressing the   approximation in
$H^{s}(\mathbb{R}^{d})$ with $s>d/2$, for example,   see
\cite{Aldroubi1,Hamm2,Butzer4,DeVore1,DeVore2,Antonio,Hamm1,Hanbin4,Qongqing,Qongqing1,Skopina2,fagep,Skopina0,Unser} and the references
therein.
In this section, we will make comparisons between the results in the present paper and the  ones in the
literature on the aspects of scope of application, flexibility of sampling and estimation techniques.

\subsection{Comparison on the applicable scope}\label{lxy}
 There are many  approximations to smooth functions in $H^{\varsigma}(\mathbb{R}^{d})$ ($\varsigma>d/2$)  in the literature such as \cite{Antonio,Hanbin4,Qongqing,Qongqing1,Skopina2,Hamm1,Hamm2}.
 These  approximations are derived from  shift-invariant spaces, and
 they only  hold for smooth functions,  but not for  the entire  space  $H^{\varsigma}(\mathbb{R}^{d})$. Clearly,   there are many functions,  including  $\phi(x_{1}, x_{2}):=B_{2}(x_{1})B_{2}(x_{2})$ in  $H^{\varsigma}(\mathbb{R}^{d})$, that  are not smooth, where $B_{2}:=\chi_{(0,1]}\star \chi_{(0,1]}$ is the cardinal
B-spline of order $2$. More precisely,  it follows from  Han \cite{Hanbin1} that  $\phi(x_{1}, x_{2})$  is in $
H^{\mu}(\mathbb{R}^{2})$ with $1<\mu<3/2.$ In \cite{fagep}, we established the   approximation to the function $f$ that satisfies
\begin{align}\label{fxu1}\begin{array}{lll}  |\widehat{f}(\xi)|\leq C(1+||\xi||_{2})^{\frac{-d-\alpha}{2}} \ \hbox{for} \ \hbox{every}\ \xi\in \mathbb{R}^{d}. \end{array}\end{align}
By Comparison \ref{duibi1}, however,  there are many functions in  $H^{\varsigma}(\mathbb{R}^{d})$  that do not satisfy \eqref{fxu1}.
That is, the above  approximations do not hold for  the entire space $H^{\varsigma}(\mathbb{R}^{d})$.
Instead  the approximation in Theorem \ref{Theorem x4} holds for the entire space $H^{\varsigma}(\mathbb{R}^{d})$.

\subsection{Comparison on the  flexibility of sampling}\label{flexible}
There exist many approximations for Sobolev spaces  available in the literature (e.g. \cite{Butzer4,Skopina0,fagep}).
But the sampling points  used for these approximations   are uniform.
 K. Hamm et.al  \cite{Hamm2,Hamm1} recently constructed  the    approximation to the univariate  functions  in  $H^{\varsigma}(\mathbb{R})\cap C^{n}(\mathbb{R})$
 by nonuniform sampling. More precisely,  the nonuniform  samples  are the values $\{f(hx_{k})\}_{k\in \mathbb{Z}}$
such that  the approximation error  is $O(h^{n})$, where the sequence  $\{x_{k}\}_{k\in \mathbb{Z}}$ is   strictly increasing   such that
$\{e^{ix_{k}x}\}_{k\in \mathbb{Z}}$ constitutes  a Riesz base for $L^2[-\pi, \pi]$.
By \cite{irregular1,Hamm1,irregular0,irregular2}, a necessary condition for a sequence
$\{x\}_{k\in \mathbb{Z}}$ to be a Riesz-type  sequence is that
there exist constants $0<q\leq Q<\infty$ such that
\begin{align}\label{toyn}\begin{array}{lll} q\leq x_{k+1}-x_{k}\leq Q. \end{array}\end{align}
A   classical sufficient condition for \eqref{toyn} is Kadec's $ 1/4$-Theorem (\cite{irregular3}), which states that
if
$|x_{k}-k|\leq1/4$ then $\{x_{k}\}_{k\in \mathbb{Z}}$ is a Riesz-type sequence.
 Instead our approximation  in \eqref{irsamapp} is conducted by the  nonuniform samples  $\{f(m^{-N}(k+\varepsilon_{k}))\}_{k\in \mathbb{Z}}$ (for the case of  dimension  $d=1$, the dilation matrix $M$ degenerates to $m$). Note that
the   sequence   $\{x_{k}\}_{k\in \mathbb{Z}}:=\{k+\varepsilon_{k}\}_{k\in \mathbb{Z}}\subseteq \mathbb{R}$ just need to satisfy \eqref{raodongxulie}.
Clearly   many  sequences satisfying \eqref{raodongxulie} are Riesz-type ones such as $\{x_{k}\}_{k\in \mathbb{Z}}=\{k+a_{0}+\frac{1}{25+k^{2}}\}_{k\in \mathbb{Z}}$
with $0<a_{0}<1/5.$
However there are also many sequences which satisfy \eqref{raodongxulie} but are not Riesz-type ones
such as $\{x_{k}\}_{k\in \mathbb{Z}}=\{k+\lambda+\theta_{k}\}_{k\in \mathbb{Z}}$ with $||\{\theta_{k}\}_{k}||_{\infty}<(1+\lambda-Q)/2$  and $\sum_{k\in \mathbb{Z}}\theta^{2}_{k}<\infty$,
and $\{x_{k}\}_{k\in \mathbb{Z}}=\{k+\lambda+\theta_{k}\}_{k\in \mathbb{Z}}$ with $||\{\theta_{k}\}_{k}||_{\infty}<\min\{\frac{1+\lambda}{2}, \frac{q-1-\lambda}{2}\}$ and $\sum_{k\in \mathbb{Z}}\theta^{2}_{k}<\infty$, where $Q$ and $q$ are as in \eqref{toyn}.
That is, the choice of  sampling sequences in the paper is  quite different from those  in  \cite{Hamm2,Hamm1}.

%
%

\subsection{Comparison on the estimation techniques}\label{ganjisong}
As mentioned in subsections \ref{lxy} and \ref{flexible},  the approximation in
the present paper is different from that in \cite{fagep} on the aspects of the applicable  scope
and  sampling flexibility. Besides the two aspects
  we next    compare the estimation techniques of the present paper with that used in  \cite{fagep}.

 (i) The estimate of $||(I-\mathcal{S}_{\phi}^{N})f||.$ The error $||(I-\mathcal{S}_{\phi}^{N})f||_{H^{s}(\mathbb{R}^{d})}$
 given in \cite[Theorem 3.2]{fagep} depends on the pointwise decay of $\widehat{f}(\xi)$ as  mentioned in \eqref{fxu1}.
However, it is very difficult to exactly   find out  the pointwise decay of $\widehat{f}(\xi)$ from the samples of $f$. Instead,  our estimation  in Lemma  \ref{Lemma 2.1} just  relays on the  global convergence  $\int_{\mathbb{R}^{d}}|\widehat{f}(\xi)|^{2}(1+||\xi||_{2}^{2})^{s}d\xi<\infty$,
and   the approximation established  in  Theorem \ref{Theorem x1} holds for all the functions in  $H^{\varsigma}(\mathbb{R}^{d})$.


(ii) The error estimate of $||(I-\mathcal{S}_{\phi;\underline{\varepsilon}}^{N})f||_{L^2}$  established in Theorem \ref{Theorem xe4} is based on   Lemma  \ref{Theorem x2}
and Lemma \ref{Theorem x3}.  The approximations in the  two lemmas were not established in \cite{fagep}.

%

%

\section{Numerical simulation}\label{diwujie}
In this section numerical simulations  are conducted  to check
the efficiency of the  approximation formula in Theorem
\ref{Theorem x4} \eqref{irsamapp}.
The perturbation sequence $\underline{\varepsilon}$ in \eqref{irsamapp} is denoted by
\begin{align}\label{szyang} \underline{\varepsilon}=\{\varepsilon_{k}:=\theta_{k}+\lambda\}_{k\in\mathbb{Z}^{d}},\end{align}
where $\theta_{k}$ is random.

\begin{figure}\label{figure4_1}
    \centering
\includegraphics[width=13cm, height=8cm]{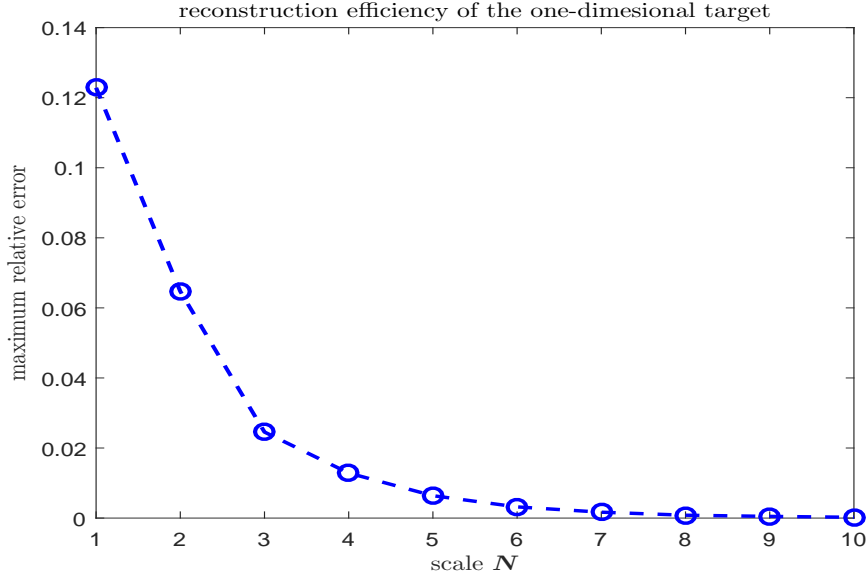}
    \caption{One-dimensional case: the maximum  relative reconstruction error vs scale $N$.}
\end{figure}

\subsection{One-dimensional case}  \label{1wei}
Let $\phi=B_{3}$, the cardinal
B-spline of order $3$ defined by
$B_{3}:=\chi_{(0,1]}\star\chi_{(0,1]}\star\chi_{(0,1]}.$
Moreover, let the target
function
\begin{align}\begin{array}{lll}\displaystyle f(x):=&\displaystyle\frac{\sin x}{x}+\frac{1}{3}B_{2}(x-2)-\frac{1}{6}\cos((x-3)^2)B_{2}(x-3)\\
&\displaystyle \ +\frac{1}{2}\cos((x-4)^2)B_{2}(x-4)-\frac{1}{2}\cos((x-5)^2)B_{2}(x-5).\end{array}\end{align}
By $\widehat{B_{2}}(\xi)=(\frac{1-e^{-\textbf{i}\xi}}{\textbf{i}\xi})^{2}$,   it is easy to check that
$f\in H^{s}(\mathbb{R})$  where $1/2<s<3/2.$  Our aim of this subsection is to  use
Theorem \ref{Theorem x4} \eqref{irsamapp} to approximate  $f$ on $[-100, 100]$.
Let  $\lambda$ in \eqref{szyang} be $1$ and the i.i.d random variables  $\{\theta_{k}\}$ obey the standard Gaussian distribution $\textbf{N}(0,1)$. Then
\begin{align} \label{bjgs} f\approx \sum^{100\times 2^{N}}_{k=-100\times 2^{N}}f(2^{-N}(k+1+\theta_{k}))\phi(2^{N}\cdot-k).\end{align}
 The  approximation  error of \eqref{bjgs}  is defined as
\begin{align}  \hbox{error}_{N}=\Big[\sum_{i\in \Lambda}\Big(f(x_{i})-\sum^{100\times 2^{N}}_{k=-100\times 2^{N}}f(2^{-N}(k+1+\theta_{k}))\phi(2^{N}x_{i}-k)\Big)^{2}\big/\sum_{j\in \Lambda}|f(x_{j})|^{2}\Big]^{1/2},\end{align}
where $\{x_{i}\}_{i\in \Lambda}=\{-100+0.01i: i=0, 1, \ldots, 20000\}$.
For each scale  $N\in \{1, 2, \ldots, 10\}$, the  approximation scheme in   \eqref{bjgs} is conducted   for $500$ trials, and the maximum of the $500$ errors is recorded in Figure 6.1. It is witnessed  in Figure 6.1 that the series in  \eqref{bjgs}
converges to $f$ on $[-100,100]$ as $N$ tends to $\infty$.

\subsection{Two-dimensional case}
\begin{figure}\label{figure42}
    \centering
\includegraphics[width=13cm, height=8cm]{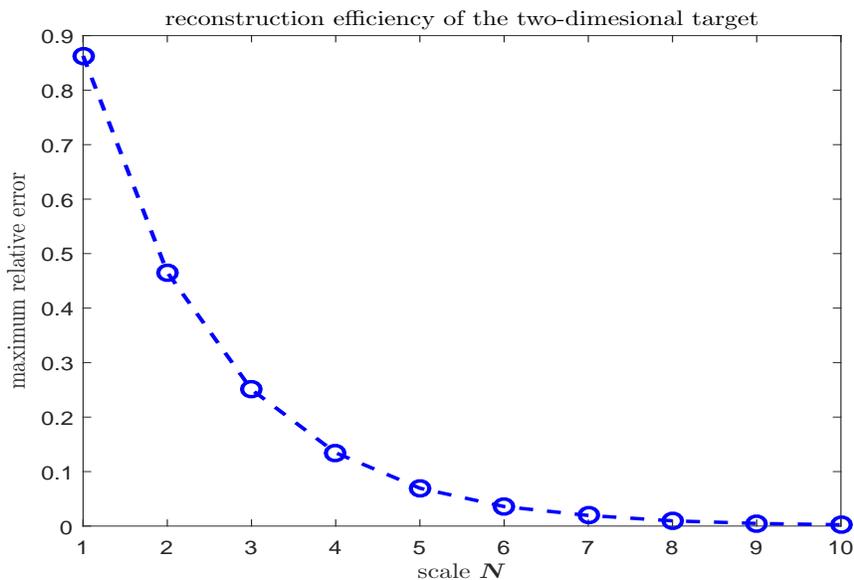}
    \caption{Two-dimensional case: the maximum  relative reconstruction error vs scale $N$.}
\end{figure}

Let $\phi(x_{1}, x_{2}):=B_{3}(x_{1})B_{3}(x_{2})$ and the target
function
$$f(x_{1}, x_{2}):=\frac{1}{(50+x^{2}_{1})(20+x^{2}_{2})}+B_{2}(x_{1})B_{2}(x_{2}).$$
Clearly,  $\phi$ is $2I_{2}$-refinable.
By \cite{Hanbin1} it is easy to check that $\phi\in H^{\mu}(\mathbb{R}^{2})$ and $f\in H^{\nu}(\mathbb{R}^{2})$,  where
 $1<\mu<5/2$ and $ 1<\nu<3/2.$
Let  $\lambda$ in \eqref{szyang} be $0.5$ and the i.i.d random variables  $\{\theta_{k}\}$ obey the standard Gaussian distribution.
Then
by  Theorem \ref{Theorem x4} \eqref{irsamapp},  we can  approximate $f$ on $[-2, 2]^{2}$ via
\begin{align} \label{approximation2}  f\approx \sum^{2^{N+1}}_{k_{1}=-2^{N+1}}\sum^{2^{N+1}}_{k_{2}=-2^{N+1}}f(2^{-N}(k+0.5+\theta_{k}))\phi(2^{N}\cdot-k),\end{align}
where $k=(k_{1}, k_{2})$.
The relative  reconstruction error is defined as
\begin{align}\label{2dreconstruction} \hbox{error}_{N}= \Big[\sum_{i\in\Lambda}|f(x_{i})-\sum^{2^{N+1}}_{k_{1}=-2^{N+1}}\sum^{2^{N+1}}_{k_{2}=-2^{N+1}}f(2^{-N}(k+0.5+\theta_{k}))\phi(2^{N}x_{i}-k)|^{2}\big/
\sum_{j\in\Lambda}|f(x_{j})|^{2}\Big]^{1/2},\end{align}
where $\{x_{i}\}_{i\in\Lambda}=\{\frac{2}{250}\ell: \ell=-250,  \ldots, 250\}\times \{\frac{2}{250}\ell': \ell'=-250,  \ldots, 250\}$
with $\times$ being the Cartesian product.
For each  scale $N\in \{1, 2, \ldots, 10\}$, the approximation scheme in \eqref{approximation2}
is conducted for $500$ trials, and  the maximum error   of the $500$ errors is  recorded in
Figure 6.2. It is witnessed  in Figure 6.2 that the series in  \eqref{2dreconstruction}
converges to $f$   as $N$ tends to $\infty$.

\section{Appendix}
\subsection{Proof of Lemma \ref{Lemma Xr}}
 We first  establish  the upper bound of $\sum_{||j||_{1}\geq
m^{J}}||j||_{1}^{-2s}$, and then    prove \eqref{op3} by  the norm  equivalence  in
$\mathbb{R}^{d}$. For $||j||_{1}\geq
m^{J}$, it is clear that there exists at least a component of $j$
such that it is not smaller than $m^{J}/d.$ Then
\begin{align} \label{zihe}\begin{array}{lllll}\displaystyle\big \{j\in \mathbb{Z}^{d}: ||j||_{1}\geq
m^{J}\big \}\subseteq \bigcup^{d}_{k=1}\big \{j=(j_{1}, j_{2},
\ldots, j_{d}): |j_{k}|\geq m^{J}/d, j_{\ell}\in \mathbb{Z},
\ell\neq k\big \}.\end{array}\end{align} By \eqref{zihe}, we have
\begin{align}\label{guji2}\begin{array}{lllll}
\displaystyle\sum_{||j||_{1}\geq
m^{J}}||j||_{1}^{-2s}&\displaystyle\leq
d\Big[\sum_{|j_{1}|\geq\textcolor[rgb]{0.00,0.07,1.00}{\lceil} m^{J}/d\rceil}\sum_{j_{2}\in
\mathbb{Z}}\cdots \sum_{j_{d}\in
\mathbb{Z}}\frac{1}{(|j_{1}|+|j_{2}|+\ldots+|j_{d}|)^{2s}} \Big],
\end{array}
\end{align}
where $\lceil x\rceil$ denotes the smallest  integer that is
larger than $x$.

Note  that the sum  in the right-hand side of \eqref{guji2}
has nothing to do with the signs of the components of $j$.
Then
\begin{align}\label{895rg}\begin{array}{lllll} \displaystyle\sum_{|j_{1}|\geq\lceil
m^{J}/d\rceil}\sum_{j_{2}\in \mathbb{Z}}\cdots \sum_{j_{d}\in
\mathbb{Z}}\frac{1}{(|j_{1}|+|j_{2}|+\ldots+|j_{d}|)^{2s}}\\
\displaystyle\leq 2^{d-1}\sum_{|j_{1}|\geq\lceil
m^{J}/d\rceil}\sum_{j_{2}\geq 0}\cdots \sum_{j_{d}\geq 0}\frac{1}{(|j_{1}|+|j_{2}|+\ldots+|j_{d}|)^{2s}}\\
\displaystyle\leq2^{d-1}\Big[\sum^{d-1}_{n=1}\tbinom{d-1}{n}I_{1,n}+I_{2}\Big],
\end{array}
\end{align}
where $\tbinom{0}{n}=0$ with $n>0$, and
\begin{align}\notag \begin{array}{lllll} \displaystyle I_{1,n}=\sum_{|j_{1}|\geq\lceil
m^{J}/d\rceil}\sum_{j_{2}\geq 1}\cdots \sum_{j_{2+n-1}\geq1}\frac{1}{(|j_{1}|+|j_{2}|+\ldots+|j_{2+n-1}|)^{2s}}, \
I_{2}=\sum^{\infty}_{j_{1}=\lceil m^{J}/d\rceil}\frac{1}{j_{1}^{2s}}.\end{array}
\end{align}


 For any $a>0 $, $N\geq1$ and $\imath>1$, it is easy
to check that
\begin{align}\label{guji4}\begin{array}{lllll}\displaystyle
\sum^{\infty}_{n=N}\frac{1}{(a+n)^{\imath}}\leq\int^{\infty}_{N-1}\frac{1}{(a+x)^{\imath}}dx=\frac{1}{\imath-1}\frac{1}{(a+N-1)^{\imath-1}}.
\end{array}\end{align}
Applying \eqref{guji4}  with  $N=1$, we obtain
\begin{align}\label{guji5}\begin{array}{lllll}
I_{1,n}
&\displaystyle\leq \prod^{n}_{l=1}\frac{1}{2s-l}\sum^{\infty}_{j_{1}=\lceil
m^{J}/d\rceil}\frac{1}{j_{1}^{2s-n}}\leq \displaystyle\sum^{\infty}_{j_{1}=\lceil
m^{J}/d\rceil}\frac{1}{j_{1}^{2s-d+1}}\prod^{n}_{l=1}\frac{1}{2s-l}
\end{array}
\end{align}
Using \eqref{guji4} again, we have
\begin{align}\label{guji5x}\begin{array}{lllll}
\displaystyle\sum^{\infty}_{j_{1}=\lceil
m^{J}/d\rceil}\frac{1}{j_{1}^{2s-d+1}}&=\displaystyle\sum^{\infty}_{j_{1}=\lceil
m^{J}/d\rceil+1}\frac{1}{j_{1}^{2s-d+1}}+\frac{1}{\lceil
m^{J}/d\rceil^{2s-d+1}}\\
&\leq \displaystyle\frac{1}{\lceil
m^{J}/d\rceil^{2s-d}}\big(\frac{1}{2s-d}+\frac{1}{\lceil
m^{J}/d\rceil}\big).
\end{array}
\end{align}
Similarly, \begin{align}\label{guji6}\begin{array}{lllll}
 \displaystyle \sum^{\infty}_{j_{1}=\lceil
m^{J}/d\rceil}\frac{1}{j_{1}^{2s}}\leq \frac{1}{\lceil
m^{J}/d\rceil^{2s-1}}\frac{1}{2s-1}+\frac{1}{\lceil
m^{J}/d\rceil^{2s}}.
\end{array}
\end{align}
 Combining \eqref{guji2}, \eqref{895rg}, \eqref{guji5},
 \eqref{guji5x} and \eqref{guji6},  we have
\begin{align}\label{guji1}\begin{array}{lllll}
 \displaystyle\sum_{||j||_{1}\geq m^{J}}||j||_{1}^{-2s}\\
 \displaystyle\leq 2^{d-1}\Big[\sum^{d-1}_{n=1}\tbinom{d-1}{n}\prod^{n}_{l=1}\frac{1}{2s-l}\frac{1}{\lceil
m^{J}/d\rceil^{2s-d}}\big(\frac{1}{2s-d}+\frac{1}{\lceil
m^{J}/d\rceil}\big)+\frac{1}{\lceil
m^{J}/d\rceil^{2s-1}}\frac{1}{2s-1}+\frac{1}{\lceil
m^{J}/d\rceil^{2s}}\Big]\\
\displaystyle\leq 2^{d-1}d^{2s-d}\Big[\sum^{d-1}_{n=1}\tbinom{d-1}{n}\prod^{n}_{l=1}\frac{1}{2s-l}(\frac{1}{2s-d}+1)+\frac{1}{2s-1}+1\Big]m^{-J(2s-d)}.
\end{array}
\end{align}
It follows from $||j||_{2}\leq ||j||_{1}\leq\sqrt{d}||j||_{2}$
that \begin{align}\label{UOH}\begin{array}{lllll}\{j\in \mathbb{Z}^d: ||j||_{2}\geq m^{J}\}\subseteq \{j\in \mathbb{Z}^d: ||j||_{1}\geq m^{J}\}.\end{array}
\end{align}
Then
\begin{align} \label{jck}\begin{array}{lllll} \displaystyle\sum_{||j||_{2}\geq m^{J}}||j||_{2}^{-2s}\leq \sum_{||j||_{1}\geq m^{J}}||j||_{2}^{-2s}\leq d^{-s}\sum_{||j||_{1}\geq m^{J}}||j||_{1}^{-2s},\end{array}\end{align}
where the first and second  inequalities are derived from \eqref{UOH} and  $||j||_{1}\leq\sqrt{d}||j||_{2}$, respectively.
Now by \eqref{guji1} and \eqref{jck},  the proof of \eqref{op3} can
be concluded.

\subsection{Proof of Lemma \ref{Theorem x2} \eqref{xiangdui05}}\label{proof of lemma3.3}
By direct computation, we get
\begin{align}\begin{array}{lllll} \label{split}
 \displaystyle \big|\langle f, m^{Nd/2}\widetilde{\phi}(M^{N}\cdot-k)
-m^{Nd/2}\widetilde{\phi}(M^{N}\cdot-k-\varepsilon_{ k})
\rangle\big|^{2}\\
= \displaystyle\frac{m^{-Nd}}{(2\pi)^{2d}} \Big|\int_{\mathbb{R}^{d}}\widehat{f}(\xi)\overline{\widehat{\widetilde{\phi}}\big((M^{T})^{-N}\xi\big)}e^{\textbf{i}(M^{T})^{-N}k\cdot\xi}(1-e^{\textbf{i}(M^{T})^{-N}\varepsilon_{ k}\cdot\xi})d\xi\Big|^{2}\\
= \displaystyle \frac{m^{-Nd}}{(2\pi)^{2d}} \Big|\int_{\mathbb{R}^{d}}\widehat{f}(\xi)(1+||\xi||_{2}^{2})^{s/2}\overline{\widehat{\widetilde{\phi}}\big((M^{T})^{-N}\xi\big)}(1+||\xi||_{2}^{2})^{-s/2}e^{\textbf{i}(M^{T})^{-N}k\cdot\xi}(1-e^{\textbf{i}(M^{T})^{-N}\varepsilon_{ k}\cdot\xi})d\xi\Big|^{2}\\
\leq
 \displaystyle\frac{m^{-Nd}}{(2\pi)^{d}}||f||^{2}_{H^{s}(\mathbb{R}^{d})}||\widehat{\widetilde{\phi}}||^{2}_{L^{\infty}(\mathbb{R}^{d})}
\int_{\mathbb{R}^{d}}(1+||\xi||_{2}^{2})^{-s}|1-e^{\textbf{i}(M^{T})^{-N}\varepsilon_{k}\cdot\xi}|^{2}d\xi\\
= \displaystyle\frac{m^{-Nd}}{(2\pi)^{d}}||f||^{2}_{H^{s}(\mathbb{R}^{d})}||\widehat{\widetilde{\phi}}||^{2}_{L^{\infty}(\mathbb{R}^{d})}\big(I_{1}(J)+I_{2}(J)\big),\end{array}
\end{align}
where the inequality is derived from the Cauchy-Schwarz inequality,
$ I_{1}(J)=\sum_{||j||_{2}\geq
m^{J}}\int_{\mathbb{T}^{d}}(1$
$+
||\xi+2j\pi||_{2}^{2})^{-s}|1-e^{\textbf{i}(M^{T})^{-N}\varepsilon_{k}\cdot(\xi+2j\pi)}|^{2}d\xi,
$ and $ I_{2}(J)=\sum_{||j||_{2}<
m^{J}}\int_{\mathbb{T}^{d}}(1+||\xi+2j\pi||_{2}^{2})^{-s}|1-e^{\textbf{i}(M^{T})^{-N}\varepsilon_{
k}\cdot(\xi+2j\pi)}|^{2}d\xi $ with  $J (>0)$ to be optimally selected.
 The
two quantities $I_{1}(J)$ and $I_{2}(J)$ are estimated as follows,
\begin{align}\label{cet6}\begin{array}{lllll}
\displaystyle I_{1}(J)&=\displaystyle\sum_{||j||_{2}\geq
m^{J}}\int_{\mathbb{T}^{d}}(1+||\xi+2j\pi||_{2}^{2})^{-s}|1-e^{\textbf{i}(M^{T})^{-N}\varepsilon_{k}\cdot(\xi+2j\pi)}|^{2}d\xi\\
&=\displaystyle4\sum_{||j||_{2}\geq
m^{J}}\int_{\mathbb{T}^{d}}(1+||\xi+2j\pi||_{2}^{2})^{-s}\big|\sin\big((M^{T})^{-N}\varepsilon_{k}\cdot(\xi+2j\pi)/2\big)\big|^{2}d\xi\\
&\leq \displaystyle4\sum_{||j||_{2}\geq
m^{J}}\int_{\mathbb{T}^{d}}(1+||\xi+2j\pi||_{2}^{2})^{-s}\big|\sin\big((M^{T})^{-N}\varepsilon_{k}\cdot(\xi+2j\pi)/2\big)\big|^{\alpha}d\xi\\
&\leq \displaystyle4||(M^{T})^{-N}\varepsilon_{
k}||^{\alpha}_{2}\sum_{||j||_{2}\geq
m^{J}}\int_{\mathbb{T}^{d}}(1+||\xi+2j\pi||_{2}^{2})^{-s}||(\xi+2j\pi)/2||^{\alpha}_{2}d\xi\\
&\leq  \displaystyle4||(M^{T})^{-N}\varepsilon_{
k}||^{\alpha}_{2}\pi^{\alpha}\sum_{||j||_{2}\geq
m^{J}}(\sqrt{d}+||j||_{2})^{\alpha}\int_{\mathbb{T}^{d}}(1+||\xi+2j\pi||_{2}^{2})^{-s}d\xi\\
&\leq  \displaystyle4||(M^{T})^{-N}\varepsilon_{
k}||^{\alpha}_{2}\pi^{\alpha}(2\pi)^{d}\sum_{||j||_{2}\geq
m^{J}}(\sqrt{d}+||j||_{2})^{\alpha}\big[1+(2\pi)^{2}(||j||_{2}-\sqrt{d})^{2}\big]^{-s}\\
&\leq \displaystyle 4||(M^{T})^{-N}\varepsilon_{
k}||^{\alpha}_{2}\pi^{\alpha}(2\pi)^{d-2s}2^{2s+\alpha}\sum_{||j||_{2}\geq
m^{J}}||j||^{-2(s-\alpha/2)}_{2}\\
&\displaystyle\leq
4\pi^{d-2s+\alpha}2^{d+\alpha}\widehat{C}(s,d)||\varepsilon_{
k}||^{\alpha}_{2}m^{-[J(2s-\alpha-d)+N\alpha]}\\
&=\displaystyle\widehat{\widetilde{C}}(s,d,\alpha)||\varepsilon_{
k}||^{\alpha}_{2}m^{-[J(2s-\alpha-d)+N\alpha]},
\end{array}
\end{align}
where $\widehat{\widetilde{C}}(s,d,\alpha)=4\pi^{d-2s+\alpha}2^{d+\alpha}\widehat{C}(s,d)$,
the second and last  inequalities  are  derived  from $\alpha\leq2$ and
Lemma \ref{Lemma Xr} \eqref{op3}, respectively. The  quantity $I_{2}$ is estimated as
follows,
\begin{align}\label{di2bufen}\begin{array}{lllll}
\displaystyle I_{2}(J)&\displaystyle\leq\sum_{||j||_{2}<
m^{J}}\int_{\mathbb{T}^{d}}(1+||\xi+2j\pi||_{2}^{2})^{-s}|1-e^{\textbf{i}(M^{T})^{-N}\varepsilon_{k}\cdot(\xi+2j\pi)}|^{2}d\xi\\
&\displaystyle\leq  (2\pi)^{d}\sum_{||j||_{2}<
m^{J}}\max_{\xi\in [0,2\pi]^{d}}|1-e^{\textbf{i}(M^{T})^{-N}\varepsilon_{k}\cdot(\xi+2j\pi)}|^{2}\\
&\displaystyle \leq 4||\varepsilon_{ k}||_{2}^{2}(2\pi)^{2d+2}m^{-2N+(2+d)J}.
\end{array}
\end{align}
 That is, $I_{1}(J)=\mbox{O}(m^{-[J(2s-\alpha-d)+N\alpha]})$ and
$I_{2}(J)=\mbox{O}(m^{-2N+(2+d)J})$. Therefore,
\begin{align}\label{laibo}
I_{1}(J)+I_{2}(J)=\mbox{O}\big(m^{-\min\{J(2s-\alpha-d)+N\alpha, \
2N-(2+d)J\}}\big).
\end{align}
It is easy to check that if choosing
$J=\frac{2-\alpha}{2s+2-\alpha}N$, then the convergence rate  in
\eqref{laibo} is optimal. Incidentally,  Lemma \ref{Lemma Xr} requires that
that   $m^{J}\geq d$.
Therefore, by  $N\geq \frac{2s+2-\alpha}{2-\alpha}\log_{m}d$  the
choice for $J=\frac{2-\alpha}{2s+2-\alpha}N$ is feasible. Now for
this choice, we have
\begin{align}\label{laibo1}
I_{1}(J)+I_{2}(J)=\mbox{O}\big(m^{-N\frac{4s+(\alpha-2)d}{2s-\alpha+2}}\big).
\end{align}
Combining  \eqref{split}, \eqref{cet6},
 \eqref{di2bufen} and \eqref{laibo1}, we have
 \begin{align}\label{yut}\begin{array}{lllll}\displaystyle
|\langle f, m^{Nd/2}\widetilde{\phi}(M^{N}\cdot-k)
-m^{Nd/2}\widetilde{\phi}(M^{N}\cdot-k-\varepsilon_{ k})\rangle|^{2}\\
\displaystyle\leq
\frac{m^{-Nd}}{(2\pi)^{d-2}}||f||^{2}_{H^{s}(\mathbb{R}^{d})}||\widehat{\widetilde{\phi}}||^{2}_{L^{\infty}(\mathbb{R}^{d})}\Big(
\widehat{\widetilde{C}}(s,d,\alpha)||\varepsilon_{
k}||^{\alpha}_{2}+4(2\pi)^{2d}||\varepsilon_{
k}||_{2}^{2}\Big)m^{-N\frac{4s+(\alpha-2)d}{2s-\alpha+2}}\\
\displaystyle\leq
\frac{m^{-Nd}}{(2\pi)^{d-2}}||f||^{2}_{H^{s}(\mathbb{R}^{d})}||\widehat{\widetilde{\phi}}||^{2}_{L^{\infty}(\mathbb{R}^{d})C_{3}(s,\alpha,d)m^{-N\frac{4s+(\alpha-2)d}{2s-\alpha+2}}}||\underline{\varepsilon}||^{2}_{\max},
\end{array}\end{align}
where
$$C_{3}(s,\alpha,d)=\widehat{\widetilde{C}}(s,d,\alpha)+4(2\pi)^{2d}.$$
 On the other hand, for any sequence $\{C_{k}\}\in
\ell^{2}(\mathbb{Z}^{d})$  we have
\begin{align}\begin{array}{lllll} \label{cet7}
||\sum_{k\in\mathbb{Z}^{d}}C_{k}\phi(\cdot-k)||^{2}_{2}&=\displaystyle(2\pi)^{-d}\int_{\mathbb{R}^{d}}|\sum_{k\in\mathbb{Z}^{d}}C_{k}e^{\textbf{i}k\cdot\xi}|^{2}|\widehat{\phi}(\xi)|^{2}d\xi\\
&=\displaystyle(2\pi)^{-d}\int_{\mathbb{T}^{d}}|\sum_{k\in\mathbb{Z}^{d}}C_{k}e^{\textbf{i}k\xi}|^{2}\sum_{\ell\in\mathbb{Z}^{d}}|\widehat{\phi}(\xi+2\ell\pi)|^{2}d\xi\\
&\displaystyle\leq||[\widehat{\phi},\widehat{\phi}]_{0}||_{L^{\infty}(\mathbb{T}^{d})}\sum_{k\in\mathbb{Z}^{d}}|C_{k}|^{2},
\end{array}
\end{align}
where the bracket product
$||[\widehat{\phi},\widehat{\phi}]_{0}||_{L^{\infty}(\mathbb{T}^{d})}$
is defined in \eqref{opq}.
 Then  it follows  from \eqref{cet7}  and  \eqref{yut}  that
\begin{align}\label{ghuilai}\begin{array}{lllll} ||(\mathcal{S}_{\phi}^{N}-\mathcal{S}_{\phi; \underline{\varepsilon}}^{N})f||^{2}_{2}\\
=\displaystyle||\sum_{k\in \mathbb{Z}^{d}} \langle f,
m^{N(d/2+s)}\widetilde{\phi}(M^{N}\cdot-k)\rangle
\phi^{s}_{N,k}-\sum_{k\in \mathbb{Z}^{d}} \langle f,
m^{N(d/2-s)}\widetilde{\phi}(M^{N}\cdot-k-\varepsilon_{k})\rangle
\phi^{s}_{N,k}||_{2}^{2} \\
 =\displaystyle||\sum_{k\in \mathbb{Z}^{d}} \langle f,
m^{Nd/2}\widetilde{\phi}(M^{N}\cdot-k)-m^{Nd/2}\widetilde{\phi}(M^{N}\cdot-k-\varepsilon_{
k})\rangle
\phi_{N,k}||_{2}^{2} \\
\leq
\displaystyle||[\widehat{\phi},\widehat{\phi}]_{0}||_{L^{\infty}(\mathbb{T}^{d})}\sum_{k\in
\mathbb{Z}^{d}}|\langle f,m^{Nd/2}\widetilde{\phi}(M^{N}\cdot-k)
-m^{Nd/2}\widetilde{\phi}(M^{N}\cdot-k-\varepsilon_{ k}) \rangle|^{2}\\
\leq
\displaystyle\frac{||f||^{2}_{H^{s}(\mathbb{R}^{d})}}{(2\pi)^{d-2}}||\widehat{\widetilde{\phi}}||^{2}_{L^{\infty}(\mathbb{R}^{d})}||[\widehat{\phi},\widehat{\phi}]_{0}||_{L^{\infty}(\mathbb{T}^{d})}
C_{3}(s,\alpha,d)
||\underline{\varepsilon}||^{2}_{\max}m^{-N[\frac{4s+(\alpha-2)d}{2s-\alpha+2}+d]},
\end{array}\end{align}
where $\phi_{N,k}=m^{Nd/2}\phi(M^{N}\cdot-k),$ and
$||\underline{\varepsilon}||_{\max}$ is defined in \eqref{fszd}.
 Now we choose
\begin{align}\label{C2}\begin{array}{lllll}  C_{2}(s, \alpha,d):=||\widehat{\widetilde{\phi}}||_{L^{\infty}(\mathbb{R}^{d})}\sqrt{\frac{||[\widehat{\phi},\widehat{\phi}]_{0}||_{L^{\infty}(\mathbb{T}^{d})}}{(2\pi)^{d-2}}
C_{3}(s,\alpha,d)}\end{array}\end{align} to conclude the proof
of \eqref{xiangdui05}.

\textbf{ Acknowledgements}: The authors  would like to thank the   reviewers for their valuable suggestions which improve the presentation of
the paper.

\end{document}